\documentclass[11pt]{amsart}
\usepackage[margin=3cm]{geometry}
\usepackage{bm}
\usepackage{amsthm,amsmath,mathtools}
\usepackage{amssymb}
\usepackage{mathrsfs}
\usepackage{hyperref,xcolor}
\usepackage[foot]{amsaddr}
\usepackage{stmaryrd,esint}
\usepackage{graphicx}

\newcommand{\R}{\mathbb{R}}
\newcommand{\N}{\mathbb{N}}
\newcommand{\Z}{\mathbb{Z}}
\newcommand{\e}{\epsilon}

\DeclareMathOperator{\area}{area}
\DeclareMathOperator{\vol}{vol}

\newtheorem{thm}{Theorem}
\newtheorem*{thm*}{Theorem}
\newtheorem*{namedthm}{\namedthmname}
\newtheorem{lem}{Lemma}
\newtheorem*{lem*}{Lemma}
\newtheorem{prop}{Proposition}
\newtheorem{coro}{Corollary}
\newtheorem*{coro*}{Corollary}

\theoremstyle{definition}
\newtheorem{defi}{Definition}

\newtheorem{rmk}{Remark}

\newcounter{namedthm}
\makeatletter
\newenvironment{named}[1]
  {\def\namedthmname{#1}%
   \refstepcounter{namedthm}%
   \namedthm\def\@currentlabel{#1}}
  {\endnamedthm}
\makeatother

\title{Volume spectrum of fiber bundles and the widths of Berger spheres}
\author{Jingwen Chen\textsuperscript{1}, Pedro Gaspar\textsuperscript{2}}
\address{\parbox{\linewidth}{
\textsuperscript{1} Department of Mathematics, University of Pennsylvania, \\
David Rittenhouse Lab,
209 South 33rd Street,
Philadelphia, PA 19104 \\ \textsuperscript{2} Facultad de Matem\'aticas, Pontificia Universidad Cat\'olica de Chile \\ Avenida Vicuña Mackenna 4860, Santiago, Chile \smallskip}}
\email{jingwch@sas.upenn.edu, pedro.gaspar@mat.uc.cl}

\begin{document}

\begin{abstract}
We establish that for a fiber bundle $\pi: E \to B$, which is a Riemannian submersion, the volume spectrum of $E$ is bounded above by the product of the volume spectrum of $B$ and the volume of the largest fiber. Specifically, we prove the following inequality:
\[
\omega_p(E,g_E) \leq \left( \sup_{b \in B} \vol_{g_E}(\pi^{-1}(b)) \right) \omega_p(B,g_B).
\]
Furthermore, we extend this result to the phase transition spectrum. In addition, we also obtain lower bounds for the isoperimetric profile of Riemannian fibrations with totally geodesic, spherical fibers in terms of the isoperimetric profile of the product of the base and a sphere. By exploiting connections between volume spectrum, least area minimal surfaces, and the isoperimetric profile, we employ these bounds to compute the low widths of Berger spheres and products of spheres. Notably, our analysis reveals that for sufficiently small $\tau$, the equatorial sphere $S^2$ in the Berger sphere $S^3_{\tau}$ (a $S^1-$bundle over $S^2(\frac{1}{2})$ with fiber length $2\pi \tau$) attains the Simon-Smith $1,2,3,4$ widths but fails to attain any lower widths, in both the Almgren-Pitts setting and the Allen-Cahn setting.
\end{abstract}

\maketitle

\section{Introduction}

Gromov \cite{gromov2002isoperimetry,gromov2006dimension,gromovsing,gromovmorsespectra} established a general framework for studying nonlinear analogs of spectral problems on Riemannian manifolds $(M^n, g)$, later extended by Guth \cite{guthminmax}. Within this framework, he introduced the $p-$widths (alternatively called volume spectrum; see Section \ref{gmtbackground} for precise definitions) as natural geometric counterparts to the eigenvalues of the Laplacian. Almgren-Pitts min-max theory \cite{Almgren,Pitts,MarquesNevesIndex,MarquesNevesRicci} says that these $p-$widths are realized as volumes of minimal hypersurfaces (with a possibly small singular set), which play a crucial role in min-max theory.

Over the past decade, the study of min-max theory and the volume spectrum has become a central topic in geometric analysis, driving remarkable progress in solving longstanding geometric variational problems. This powerful framework has yielded profound results, including the complete proof of the Willmore conjecture by Marques-Neves \cite{MarquesNevesWillmore}; the resolution of Yau conjecture on the existence of infinitely many minimal hypersurfaces by Song \cite{Song} (see also Marques-Neves \cite{MarquesNevesRicci}, Irie-Marques-Neves \cite{IrieMarquesNeves}, Chodosh-Mantoulidis \cite{ChodoshMantoulidis}, Li \cite{YangyangInfinitely} and Zhou \cite{Zhou} for the case of generic or positively curved metrics); and more recently, a new proof of the Smale conjecture for all lens spaces by Ketover-Liokumovich \cite{ketover2025smale}; the proof of the existence of four distinct embedded minimal two-spheres in the three dimensional sphere with a bumpy metric or a metric with positive Ricci curvature by Wang-Zhou \cite{wangzhouexistence} (we mention also the follow up work on the existence of embedded tori by Chu-Li \cite{chu2024existence} and Li-Wang \cite{li2024existence}).

Despite the rich applications of the volume spectrum, explicit computations of the $p$-widths remain relatively scarce. For dimension $n=2$, the $p$-widths have been computed for $S^2$ and $\R P^2$ in \cite{CMSurfaces, marx2025p}, respectively. Very recently, low widths of hyperbolic surfaces have been computed by Lima \cite{lima2025first} (see also \cite{mantoulidis2025almgren}), and certain widths of convex polygons have been computed by Chodosh-Cholsaipant \cite{chodoshCholsaipant}. However, known results become significantly more limited in higher dimensions. For $n \geq 3$, only certain low-width cases have been computed for specific classes of Riemannian manifolds, including the round spheres, balls, projective spaces, lens spaces, see e.g. \cite{nurser2016low, CPChuLi, donato2022first, batista2022min, luna2019compact, batista2023short}.

The computation of widths presents substantial challenges even in the simplest cases, the $1-$width. For general Riemannian manifolds $(M^n, g)$, determining the first width remains particularly difficult. For $3 \leq n \leq 7$, Mazet and Rosenberg \cite{mazetrosenberg} established a classification theorem: any least-area minimal hypersurface must be either stable or obtained through min-max construction (see also \cite{zhoupositive} for the case of manifolds of positive Ricci curvature and any dimension $n\geq 3$). Building on this result, we observe that when stable minimal hypersurfaces are absent, any minimal hypersurface bounding an isoperimetric domain must have least area among all minimal hypersurfaces (see Lemma \ref{leastarea_isoperimetric} for a detailed statement). This observation suggests an alternative computational approach -- determining the first width through solutions of the corresponding isoperimetric problem.\medskip

In this paper, we investigate the relationship between the volume spectra of Riemannian bundles and the volume spectra of their base spaces and fibers. Our first main result is:

\begin{named}{Theorem A} \label{thm:width_fibration}
Let $\pi\colon E \to B$ be a fiber bundle, with $E$ and $B$ compact, and endowed with Riemannian metrics $g_E$ and $g_B$ that make $\pi$ a Riemannian submersion. Then the volume spectra of $(E,g_E)$ and $(B,g_B)$ satisfy:
    \[
    \omega_p(E,g_E) \leq \left( \sup_{b \in B} \vol_{g_E}(\pi^{-1}(b)) \right) \omega_p(B,g_B).
    \]
\end{named}

We provide two proofs for this inequality: one using only tools from Geometric Measure Theory and the \emph{fiber integration map} for differential forms (see \cite{GHV}):
\[
I_\pi ={\textstyle \fint_F}\colon \Omega^{k+\dim B}(E) \to \Omega^k(B),
\]
and an independent proof using the relation between the volume spectrum and the \emph{phase transition spectrum} \cite{GasparGuaraco,Dey} and the following inequality:

\begin{named}{Theorem B} \label{thm:phasetransition}
Let $\pi\colon E \to B$ be as in \ref{thm:width_fibration}. Then for every $\epsilon>0$, the phase transition spectra of $(E,g_E)$ and $(B,g_B)$ satisfy:
    \[
    c_\epsilon(p;E,g_E) \leq \left( \sup_{b \in B} \vol_{g_E}(\pi^{-1}(b)) \right) c_\epsilon(p;B,g_B).
    \]
\end{named}

The numbers $c_\e(p;M,g)$ are critical values for the \emph{Allen-Cahn energy functional}:

\[E_{\epsilon}(u):= \int_{M} \left(\frac{\epsilon}{2} |\nabla_{g} u|^2 + \frac{1}{\epsilon} W(u)\right) d\mu_{g}, \ \ u \in W^{1,2}(M).\]
The asymptotic behavior of this functional as $\e \to 0^+$ was studied by Modica and Mortola \cite{Modica,MM} within the framework of $\Gamma$-convergence. Roughly speaking, they proved that the Allen–Cahn energy functional $\Gamma$-converges to the perimeter functional (times a constant depending only on the real-valued double-well potential $W$), a generalization of the $(n-1)-$dimensional volume defined on the space of sets of finite perimeter. This result allows one to use the Allen-Cahn energy functional to study the area functional, and conversely. We refer to Section \ref{subsec:ac} for the precise definitions related to the Allen-Cahn functional and the volume spectrum.

\ref{thm:width_fibration} and \ref{thm:phasetransition} provide upper bounds for the widths of the fiber bundle. Obtaining lower bounds is considerably more challenging. We study the sphere bundle $\pi\colon(E,g_E) \to (B,g_B)$, with totally geodesic fibers $E_b = \pi^{-1}(b)$ isometric to the round sphere $S^r$. By employing fiberwise symmetrization techniques, we compare the isoperimetric profiles of $E$ and $B\times S^r$. We will denote by $S^r$ the $r$-dimensional unit sphere, and by $S^r(\rho)$ the $r$-dimensional sphere of radius $\rho$.

Our second main result is:

\begin{named}{Theorem C} \label{isoperimetric bounds}
For $\rho>0$, let $g_E^\rho$ be the Riemannian metric on $E$ obtained by scaling the fibers of $\pi$ by $\rho^2$. The isoperimetric profiles of $(E,g_E^\rho)$ and $B\times S^r(\rho)$ satisfy
\[
I_{(E,g_E^\rho)} (t) \geq I_{(B\times S^r(\rho))} (t), \qquad \text{for all} \ t \in (0,\vol(E,g_E^\rho)).
\]   
\end{named}

For a more precise formulation, see Theorem \ref{thm:lower_bound}. \medskip

Based on Lemma \ref{leastarea_isoperimetric}, we can use these methods to get lower bounds for the first width of $E$ in some specific cases. As an application of our main results, we compute the low-widths of Berger spheres and products of spheres.

The Berger spheres $S^3_\tau$ are important examples of homogeneous $3$-manifolds which can be roughly described as a smooth family of deformations of the round sphere -- see Section \ref{subsec:berger} for the precise description of this family of metrics. From the min-max perspective, our motivation comes from recent work by Ambrozio-Marques-Neves \cite{AmbrozioMarquesNeves1,AmbrozioMarquesNeves2}, in which they study rigidity phenomena for the volume spectrum, and characterize \emph{surface Zoll metrics} in terms of the Simon-Smith variant of the widths. As an application of \ref{thm:width_fibration} and \ref{isoperimetric bounds}, we will show that:

\begin{named}{Corollary D} \label{corollary Berger} For $\tau \in (0,\frac{1}{\pi})$, the first $3$ widths of the Berger sphere $S^3_\tau$ are achieved by the minimal Clifford tori in $S^3_\tau$. Moreover, if $\tau$ is sufficiently small, then the equatorial spheres in $S^3_\tau$ cannot achieve low $p$-widths $\omega_p(S^3_\tau)$.
\end{named} 
We observe that this Corollary contrasts sharply with the Simon-Smith setting, in which equatorial spheres achieve the corresponding critical values $\sigma_i(S^3_\tau)$, for $i=1,2,3,4$.\medskip

In the case of products of spheres $S^{n_1}\times S^{n_2}(a)$, we compute the first few widths as follows:

\begin{named}{Corollary E} \label{corollary product spheres}
Suppose $n_1 \geq n_2 \geq 2$ and $a>0$. There exist positive numbers $a_-(n_1,n_2)<a_+(n_1,n_2)$ such that:
\begin{enumerate}
    \item[(a)] if $a\in (0,a_-(n_1,n_2)]$, then $$\omega_1(S^{n_1}\times S^{n_2}(a)) = \ldots = \omega_{n_1+1}(S^{n_1}\times S^{n_2}(a)) = \mathrm{vol}(S^{n_1-1} \times S^{n_2}(a)).$$
    \item[(b)] if $a\in [a_-(n_1,n_2),\infty)$, then $$\omega_1(S^{n_1}\times S^{n_2}(a)) = \ldots = \omega_{n_2+1}(S^{n_1}\times S^{n_2}(a)) = \mathrm{vol}(S^{n_1} \times S^{n_2-1}(a)).$$
\end{enumerate}
Moreover, if $n_1 + n_2 \leq 7$, then we can choose $a_-(n_1,n_2) = a_+(n_1,n_2) = \frac{|S^{n_2}|/|S^{n_2-1}|}{|S^{n_1}|/|S^{n_1-1}|}$.
\end{named}

We refer to subsection \ref{width product spheres} for the explicit description of the constants $a_-(n_1,n_2)$ and $a_+(n_1,n_2)$, and observe that this result recovers the computation of the low widths by Batista-Martins \cite{BM} and H. Chen \cite{ChenStability,ChenH} under the same assumptions on the dimensions $n_1$ and $n_2$ and the radius $a$.

Our work provides new tools to estimate the min-max widths, and we expect that our results could apply to other geometrically interesting spaces, such as the unit tangent bundle to compact Riemann surfaces and other homogeneous spaces.

\subsection*{Organization}

The paper is organized as follows.

In Section $2$, we review the background on geometric measure theory and Almgren-Pitts min-max theory, Allen-Cahn equation and the phase transition spectrum, and Fiber integration map, and give a proof of the relation between least area minimal surfaces and isoperimetric domains.

In Section $3$, we use the fiber integration map to connect the sweepouts of the fiber bundle and the sweepouts of the base, hence proving  \ref{thm:width_fibration}. 

In Section $4$, we extend our results to the phase transition spectrum and prove \ref{thm:phasetransition}.

In section $5$, we apply the fiberwise symmetrization method to study the isoperimetric profile of sphere bundles, and prove \ref{isoperimetric bounds}.

Finally, in Section $6$, we apply our findings to compute the low widths of some specific Riemannian manifolds, the Berger spheres and the product of spheres.

\subsection*{Acknowledgements}

We would like to thank Andr\'e Neves for inspiring discussions and suggestions. We also thank Xin Zhou for helpful discussions. PG was supported by ANID (Agencia Nacional de Investigaci\'on y Desarrollo) FONDECYT Iniciaci\'on grant.

\section{Preliminaries}

\subsection{Currents and sweepouts} \label{gmtbackground}

In this subsection, we recall the definitions and preliminary results about geometric measure theory and Almgren-Pitts min-max theory. Standard references are the lecture notes of Simon \cite{simonbook} and the classic book by Federer \cite{Federer}.

Let $(M,g)$ be an orientable compact Riemannian manifold. Let $\Omega^k(M)$ denote the space of smooth $k-$forms in $M$. A $k-$dimensional current (briefly called an $k-$current) in $M$ is a continuous linear functional on $\Omega^k(M)$. The set of such $k-$currents will be denoted by $\mathcal{D}_k(M)$. Note that there is a current $[\![M]\!]$ induced by integration over the total space $M$. We will consider the following spaces of currents in $M$:

\begin{itemize}
    \item $\mathcal{R}_k(M)$: the space of integer rectifiable $k-$currents (see \cite[4.1.24]{Federer}).
    \item $\mathcal{F}_k(M)$: the space of integral \emph{flat $k-$chains} in $M$, defined by \[\mathcal{F}_k(M)=\{T = R + \partial S: R \in \mathcal{R}_k(M), S \in \mathcal{R}_{k+1}(M)\}.\]
    \item $\mathbf{I}_k(M)$: the space of $k-$dimensional integral currents, namely  \[\{T \in \mathcal{R}_k(M) \colon \partial T \in \mathcal{R}_{k-1}(M)\}.\]
    \item $\mathcal{Z}_k(M)$: the space of $k-$cycles in $M$, that is those $T \in \mathbf{I}_k(M)$ with $\partial T = 0$.
\end{itemize}

We also recall the definition of modulo $2$ flat chains from \cite[4.2.26]{Federer} in a compact manifold $M$ (see also \cite{MarcheseStuvard}). If $T_1$ and $T_2$ are flat $k$-chains, then we say $T_1$ are $T_2$ \emph{congruent modulo 2} if there exists a flat $k$-chain $Q$ such that $T_1-T_2=2Q$. We will follow Federer's notation and write $(T)^2$ for the congruence class of a flat chain $T$. Since the boundary operator preserves these congruence classes, it induces an operator between modulo $2$ chains. 

The spaces we will work within this paper are:

\begin{itemize}
    \item $\mathcal{R}_k(M;\Z_2)$: the space congruence classes of rectifiable $k$-currents.
    \item $\mathbf{I}_k(M;\Z_2)$: the space of $k-$dimensional mod $2$ flat chains $M$, that is \[\mathbf{I}_k(M;\Z_2) = \{\tau \in \mathcal{R}_k(M;\Z_2) \colon \partial \tau \in \mathcal{R}_{k-1}(M;\Z_2)\}.\]
    \item $\mathcal{Z}_k(M; \Z_2)$: the space of mod $2$ $k-$cycles, that is $\tau \in \mathbf{I}_k (M;\Z_2)$ with $\partial \tau = 0 \in \mathcal{R}_{k-1}(M;\Z_2)$.
    \item the closure $\mathcal{V}_k(M)$, in the weak topology, of the space of $k-$dimensional rectifiable varifolds in $\R^L$ with support contained in $M$.
\end{itemize}

Given $T \in \mathbf{I}_k(M)$ or $\mathbf{I}_k(M,\Z_2)$, the associated integral varifold and Radon measure in $M$ are denoted by $|T|$ and $\|T\|$ respectively. Given $V \in \mathcal{V}_k(M)$, $\|V\|$ denotes the Radon measure in $M$ associated with $V$. The above spaces come with several relevant metrics. To introduce them, recall that for a differential form $\eta$ defined on a Riemannian manifold $(M,g)$, we denote by $\|\eta(x)\|$ the \emph{comass} of the alternating $k$-form in $\eta(x)\in \Lambda^kT^*M$ with respect to the inner product induced by $g_x$ on alternating bundles, that is:
\[\|\eta(x)\| = \sup\left\{ \left\langle\eta(x),\xi\right\rangle\colon \xi \in \Lambda^kT_xM, \ \xi\ \text{simple}, \ |\xi|_{g_x}\leq 1\right\}.\]

The \emph{mass} of $T \in \mathbf{I}_k(M)$, defined by
\[\mathbf{M}(T) = \sup \{T(\phi): \phi \in \Omega^k(M), \|\phi\| \leq 1\},\]
induces the metric $\mathbf{M}(S,T) = \mathbf{M}(S-T)$ on $\mathbf{I}_k(M)$. 

The \emph{flat metric} is defined by
\[\mathcal{F}(S,T) = \inf \{\mathbf{M}(P)+\mathbf{Q}: S-T=P+\partial Q, P \in \mathbf{I}_k(M), Q \in \mathbf{I}_{k+1}(M)\},\]
for $S,T\in \mathbf{I}_k(M)$. We also use $\mathcal{F}(T) = \mathcal{F}(T,0)$. Note that
\begin{equation*}
\mathcal{F}(T) \leq \mathbf{M}(T) \text{ 
 for all } T \in \mathbf{I}_k(M).   
\end{equation*}

The $\mathbf{F}-$metric on $\mathcal{V}_k(M)$ is defined as in Pitts book \cite{Pitts}:
\[\mathbf{F}(V,W) = \sup \{V(f)-W(f): f\in C_c(G_k(\R^L)), |f|\leq 1, \text{Lip}(f) \leq 1\},\]
for $V,W \in \mathcal{V}_k(M)$. The $\mathbf{F}-$metric induces the varifold weak topology on $\mathcal{V}_k(M)$, and it satisfies
\[\mathbf{F}(|S|,|T|) \leq \mathbf{M}(S-T) \text{    for all S,T} \in \mathbf{I}_k (M).\]

Finally, the $\mathbf{F}-$metric on $\mathcal{R}_k(M)$ is defined by
\begin{equation} \label{F-metric}
    \mathbf{F}(S,T) = \mathcal{F}(S-T) + \mathbf{F}(|S|,|T|).
\end{equation}

We assume that $\mathbf{I}_k(M)$ and $\mathcal{Z}_k(M)$ both have the topology induced by the flat metric. If endowed with the $\mathbf{F}-$metric, we will denote them by $\mathbf{I}_k(M;\mathbf{F})$ and $\mathcal{Z}_k(M;\mathbf{F})$, respectively.\medskip

For mod $2$ flat chains $T \in \mathcal{R}_k (M)$, its modulo 2 mass $\mathbf{M}^2 (T)$ is defined as the least $t \in \R \cup \{+\infty\}$ such that for every $\e > 0$, there exists a rectifiable current $R \in \mathcal{R}_{k}(M)$ with
\begin{equation} \label{mod v mass}
    \mathcal{F}^2 (T-R) \leq \e,\quad  \text{and} \quad \mathbf{M}(R) \leq t+ \e.
\end{equation}

In $1962$, Almgren in his thesis \cite{Almgren} proved the following natural isomorphism:
\[\pi_k(\mathcal{Z}_n(M^{n};\Z_2), 0) \cong H_{k+n-1}(M^{n};\Z_2),\]
and later, it was shown that $\mathcal{Z}_n(M^{n};\Z_2)$ is weakly homotopic to $\R P^{\infty}$, see \cite{MarquesNevesMultiplicity}. Let's denote the generator of $H^1(\mathcal{Z}_n(M^{n};\Z_2); \Z_2) = \Z_2$ by $\overline{\lambda}$.

\begin{defi}[\cite{Pitts,MarquesNevesMultiplicity}]
Given a closed Riemannian manifold $(M^{n},g)$, a continuous map $\Phi: X \to \mathcal{Z}_n(M^{n};\mathbf{F};\Z_2)$ is called a \emph{p-sweepout}, provided that $X$ is a finite dimensional compact cubical complex, and $\Phi^* (\overline{\lambda}^p) \neq 0$ in $H^p(X,\Z_2)$. 

We denote by $\mathcal{P}_p(M)$ the set of all $p-$sweepouts that are continuous in the flat topology and have no concentration of mass. This means that for the quantity
\[\textbf{m}(\Phi, r) = \sup \{\|\Phi(x)\| (B_r(p)): x \in X, p \in M\},\]
we require $\lim\limits_{r \to 0} \textbf{m} (\Phi, r) = 0$, where $B_r(p)$ denotes the open geodesic ball in $M$ of radius $r$ centered at $p \in M$.
\end{defi}

Next, we state some essential results about the volume spectrum of a closed Riemannian manifold $(M^n,g)$.

\begin{defi}
The Almgren-Pitts (min-max) \emph{p-width} $\omega_p(M,g)$ is
\[\omega_p(M,g):= \inf\limits_{\Phi \in \mathcal{P}_p(M)} \sup\limits_{x \in X} \mathbf{M}(\Phi(x)).\]
\end{defi}

Liokumovich-Marques-Neves \cite{LMNWeyl} proved that the volume spectrum satisfies a Weyl law:

\begin{thm}
There exists a constant $a(n)>0$ such that for every compact Riemannian manifold $(M^{n},g)$ with (possibly empty) boundary, we have
\[\lim\limits_{p \to \infty} \omega_p(M) p^{-\frac{1}{n}} = a(n) \vol (M)^{\frac{n-1}{n}}.\]
\end{thm}

In addition, we have the following main existence result in the min-max theory.

\begin{thm} \label{widths are achieved} \emph{(\cite{MarquesNevesIndex}, \cite[Theorem C]{Zhou}, \cite[Theorem 1.3]{Yangyang}, \cite[Theorem 1.5]{Dey})} For every $p \in \N$, there exists a stationary integral varifold $V$ with $\mathrm{spt}\|V\| = \Sigma$ such that
    \begin{enumerate}
        \item[(i)] $\omega_p(M,g) = \|V\|(M)$,
        \item[(ii)] $\Sigma$ is a minimal hypersurface with \emph{optimal regularity}, namely it is smooth away from a singular set of Hausdorff dimension at most $(n-8)$.
        \item[(iii)] The Morse index of $\Sigma$ satisfies $\mathrm{ind}(\Sigma) \leq p$.
    \end{enumerate}
\end{thm}

\subsection{Least area minimal surfaces and isoperimetric domains} In our applications, we would like to study certain minimal hypersurfaces that bound isoperimetric domains.

We recall that least area minimal hypersurfaces in closed Riemannian manifolds $(M^n,g)$ were characterized by Mazet-Rosenberg in \cite{mazetrosenberg}, for $3\leq n \leq 7$ (see also Zhou \cite{zhoupositive} for $n \geq 3$ assuming $\mathrm{Ric}_g>0$). Roughly stated, their results say that such minimal hypersurfaces are either stable, or separating min-max minimal hypersurfaces with Morse index 1 that can be obtained from the fundamental class of $M$ in $H_n(M,\Z_2) \simeq {\pi_1}(\mathcal{Z}_{n-1}(M,\Z_2),0)$. We refer to Theorem A in \cite{mazetrosenberg} for details.

Using this characterization, we observe that, in the absence of stable minimal hypersurfaces, a minimal hypersurface that bounds an isoperimetric domain must be of least area. More generally, we have:

\begin{lem} \label{leastarea_isoperimetric}
Let $(M^n,g)$, $3 \leq n \leq 7$, be a closed oriented Riemannian manifold. Suppose there exists a minimal hypersurface $\Gamma$ in $(M,g)$ such that $\Gamma = \partial D$ for some isoperimetric domain $D \subset M$ with $\vol_g(D) \in (0,\vol_g(M))$.

Suppose
    \begin{equation} \label{unstable_leastarea}
        |\Gamma| < \inf \left(\hspace{-20pt}\begin{array}{c}\{|\Sigma| \ \colon \ \Sigma \subset (M,g) \ \text{closed, stable, orientable, minimal hypersurface}\} \\ \ \qquad \cup \ \{2|\Sigma| \ \colon \ \Sigma \subset (M,g) \ \text{closed, stable, non-orientable, minimal hypersurface}\} \end{array} \right),
    \end{equation}
where the infimum is interpreted as $+\infty$ if $(M,g)$ admits no stable minimal hypersurfaces. Then $\Gamma$ is a least area minimal hypersurface, namely $|\Gamma|\leq | \Sigma|$ for any closed, smoothly embedded minimal hypersurface $\Sigma$ in $(M,g)$. In particular $W(M,g) = \omega_1(M,g)=|\Gamma|$, where $W(M,g)$ denotes the Almgren-Pitts width of $(M,g)$.
\end{lem}

\begin{proof}
We argue by contradiction. If we suppose otherwise, then the least area minimal hypersurface $\Sigma$ in $(M,g)$ must satisfy $|\Sigma| < |\Gamma|$, and by \eqref{unstable_leastarea} it is unstable. By \cite[Proposition 11]{mazetrosenberg}, $\Sigma$ must be homologous to $0$, that is $\Sigma = \partial \Omega$ for some domain $\Omega \subset M$. Since we assume $\Gamma = \partial D$ for an \emph{isoperimetric domain} $D$ with volume $v$, we have $\vol_g(\Omega) \neq \vol_g(D)$. 

Choose a 1-parameter sweepout of $(M,g)$ associated to the fundamental class of $M$ in $H_n(M,\Z_2) \simeq {\pi_1}(\mathcal{Z}_{n-1}(M,\Z_2),0)$, say $\{\Sigma_t=\partial \Omega_t\}_{t \in [-1,1]}$, such that:

\begin{enumerate}
    \item[(i)] the sweepout is optimal: $\Sigma_0 = \Sigma$ and $|\Sigma_t| < |\Sigma|$ for $t \neq 0$, and
    \item[(ii)] $t\in[-1,1] \mapsto \vol_g(\Omega_t)$ is continuous function. 
\end{enumerate}
Here we choose $\Omega_t$ as a lift of the nontrivial path $t \mapsto \Sigma_t$ to the space of Caccioppoli sets $\mathcal{C}(M)$, so that $\Omega_{(-1)} = 0$ and $\Omega_1 = M$. The existence of such a sweepout follows, e.g., from \cite[Proposition 17]{mazetrosenberg}.

On the other hand, this directly leads to a contradiction. To see this, note that by continuity there is $t \in (-1,1)$ with $t \neq 0$ such that $\vol_g(\Omega_t) = \vol_g(D)$. Consequently:
\[
|\Gamma| > |\Sigma| = |\Sigma_0| > |\Sigma_t| = |\partial \Omega_t| \geq |\partial D| = |\Gamma|.
\]
Here, the first inequality follows from our contradiction assumption, the second inequality follows from the choice of the sweepout, and the last inequality follows from the isoperimetry of $D$.
\end{proof}

\begin{rmk}
We note that condition \eqref{unstable_leastarea} is trivially met if $(M,g)$ contains no stable, orientable minimal hypersurfaces (e.g. if it has positive Ricci curvature) and contains no non-orientable, closed, minimal hypersurfaces (e.g., if $H_{n-1}(M,\Z_2)$ is trivial).    
\end{rmk}

\subsection{Allen-Cahn equation and the phase transition spectrum} \label{subsec:ac}

In this subsection, we consider a closed Riemannian manifold $(M,g)$ and recall some basic facts and results about the \emph{(elliptic) Allen-Cahn equation}:
\begin{equation} \label{AC}
\Delta_g u -  {\textstyle \frac{1}{\e^2}}W'(u)=0 \quad \text{on} \ (M,g). \tag{AC}
\end{equation}
and the phase transition spectrum.

\begin{defi}
A function $W \in C^{\infty}(\R)$ is a \emph{(symmetric) double-well potential} if:
\begin{enumerate}
    \item[(1)] $W$ is nonnegative and vanishes precisely at $\pm1$;
    \item[(2)] $W$ satisfies $W'(0) = 0$, $W''(0) \neq 0$, and $sW'(s) < 0$ for $|s| \in (0, 1)$;
    \item[(3)] $W''(\pm1)>0$;
    \item[(4)] $W(s) = W(-s)$ for all $s \in \R$.
\end{enumerate}
\end{defi}

The standard example of a double-well potential is the function $W(t) = \frac{1}{4}(1 - t^2)^2$. 

\begin{defi}
We define the \emph{Allen-Cahn energy} on $(M,g)$ by:
\[E_{\epsilon}(u):= \int_{M} \left(\frac{\epsilon}{2} |\nabla_{g} u|^2 + \frac{1}{\epsilon} W(u)\right) d\mu_{g}, \ \ u \in W^{1,2}(M),\]
where $d\mu_g$ is the volume measure with respect to $g$. 
\end{defi}

One can check that $u$ is a critical point of $E_{\epsilon}$ if and only if $u$ (weakly) solves \eqref{AC}. We write $\sigma = \int_{-1}^1 \sqrt{W(t)/2} dt$. This is the energy of the \emph{heteroclinic solution} $\mathbb{H}_\e(t)$ of \eqref{AC} on $\R$, that is, the unique bounded solution in $\R$ (modulo translation) such that $\mathbb{H}_\e(t) \to \pm 1$ when $t \to \pm \infty$. We refer to \cite[Section 1.3]{ChodoshMantoulidis} for more information on this one-dimensional solution.

As observed in the introduction, the Allen-Cahn energy is closely related to isoperimetric problems through $\Gamma$-convergence. For every $m \in (-\vol(M,g),\vol(M,g))$, let
    \[
        V_m(M,g) = \left\{ u \in W^{1.2}(M) \colon \int_Mu\,d\mu_g = m\right\}.
    \]
It follows from the direct method of the calculus of variations and elliptic regularity that the infimum of $E_\e$ over functions in $V_m(M,g)$ is achieved by a smooth solution $u_{\e,m} \in V_m(M,g)$ of the nonhomogeneous Allen-Cahn equation $\e\Delta_gu - W'(u)/\e = \lambda_m$, where $\lambda_m$ is a Lagrange multiplier associated to the integral constraint.

If follows from the classical work of Modica \cite{Modica} that $u_{\e,m}$ subconverges as $\e\downarrow 0$ to a function $u_m \in BV(M,g)$ (in the $BV$ sense) which satisfies $W(u_m) = 0$ a.e. on $M$ and for which $M\cap \{u_m=-1\}$ is a solution of the \textit{isoperimetric problem}
\[
\min\left\{\ \mathrm{Per}_{(M,g)}(E)\ \colon\ E \in \mathcal{C}(M)\ ,\ \vol_g(E) = \frac{\vol(M,g)-m}{2} \ \right\},
\]
where $\mathrm{Per}_{(M,g)}(E) = \|D\chi_E\|(M)$ denotes the \emph{perimeter} of the Caccioppoli set $E$ (see \cite[\S 14]{simonbook}). The value above coincides with $I_{(M,g)}(\frac{\vol(M,g)-m}{2})$, where $I_{(M,g)}$ is the isoperimetric profile of $(M,g)$, and satisfies:
    \[ I_{(M,g)}\left(\frac{\vol(M,g)-m}{2}\right) = \frac{1}{2\sigma}\lim_{\e \downarrow 0} E_\e(u_{\e,m}).\]
    
We now briefly state the definition of the phase transition spectrum of $(M,g)$, following the papers \cite{Guaracominmax,GasparGuaraco} where further details can be found. Let $\text{Ind}_{\Z_2}$ denote the the Fadell-Rabinowitz cohomological $\Z_2$-index (see \cite[Section 3.1]{GasparGuaraco} for detailed definition). Let $X$ be a $C^2$ Hilbert manifold which is also a free $\Z_2-$space. For each $p \in \N$, consider the family
\begin{equation*}
\mathcal{F}_p = \{A \subset W^{1,2}(M): A \text{ compact, symmetric}, \text{ Ind}_{\Z_2}(A) \geq p+1\}.
\end{equation*}
One easily verifies that $\mathcal{F}_p$ is a $p-$dimensional $\Z_2-$cohomological family, in the sense defined in \cite[Section 3]{ghoussoub1991location}.

Define $X = W^{1,2}(M) \setminus \{0\}$ and let $\mathcal{F}_p$ be the cohomological family defined as above. We define the $p$-th $\e-$Allen-Cahn width $c_{\e}(p;M,g)$ as:
\begin{equation*}
c_{\e}(p;M,g) := \inf\limits_{A \in \mathcal{F}_p} \sup\limits_{x \in A} E_\e (x), \text{  for  }p \in \N.
\end{equation*}

When the ambient dimension $3 \leq n \leq 7$, it was proved by \cite{GGweyl} that $\lim\limits_{\e \to 0^+} c_{\e}(p;M,g)$ exists.

In \cite{Dey}, Dey compared the Almgren-Pitts width $\omega_p$ and the limit of the $\e-$Allen-Cahn width. 
By \cite[Theorem1.2]{Dey}, the Almgren-Pitts width $\omega_p$ can be obtained as the limit of the $\e$-Allen-Cahn widths $c_{\e}(p;M,g)$, that is
\begin{equation}\label{AP-AC widths}
\frac{1}{2\sigma}\lim_{\epsilon \to 0^+}c_{\e}(p;M,g)= \omega_p(M,g), \forall p \in N.
\end{equation}

\subsection{Fiber integration}
Let $\pi: E \to B$ be a smooth fiber bundle. In this subsection, we will recall the construction of \emph{Thom’s fiber integration map}:
\[
I_\pi ={\textstyle \fint_F}\colon \Omega^{k+r}(E) \to \Omega^k(B),
\]
and refer to \cite[Chapter VII]{GHV} for further details. 

Hereafter, we will denote:
\begin{itemize}
    \item by $E_b = \pi^{-1}(b)$ the fiber at point $b \in B$;
    \item the dimensions of the base and the fibers by $\dim B = n$ and $\dim E_b = r$, respectively;
    \item by $V_E \to E$ the \emph{vertical bundle} of $E\to B$. This is a smooth vector subbundle of $TE$ of rank $r$, whose fiber over $e \in E$ is $(V_E)_e=\ker(D\pi(e)) = T_e(E_{\pi(e)}) \subset T_eE$, the tangent space to the fiber through $e$.
\end{itemize} 

Let $\omega \in \Omega^{k+r}(E)$ and consider the smooth bundle morphism $\Phi_\omega \colon \Lambda^{r}V_E \to \Lambda^k\ T^*B$ (over $\pi \colon E \to B$) given by,
\[
\left\langle \Phi_\omega(e;v_1\wedge \ldots \wedge v_r),\xi_1\wedge\ldots\wedge \xi_k\right\rangle = \omega_e(\bar\xi_1,\ldots,\bar\xi_k,v_1,\ldots,v_r),
\]
for $e\in E$, $v_1,\ldots, v_r \in (V_E)_e = T_e(E_{\pi(e)})$, $\xi_1,\ldots,\xi_k \in T_{\pi(e)}B$, and $\bar \xi_1,\ldots, \bar \xi_k \in T_eE$ such that $D\pi(e)\bar \xi_i = \xi_i$. It follows from $\dim \ker(D\pi(e)) = \dim((V_E)_e) = r$ that this definition is independent of the choice of the vectors $\bar \xi_i$, and hence $\Phi_\omega(e;v_1\wedge \ldots \wedge v_r)$ is a well-defined element of $\Lambda^kT^*_{\pi(e)}B$. The restriction of $\Phi_\omega$ to $\left(\Lambda^k V_E\right)|_{E_b} = \Lambda^k TE_b$ defines a smooth $\Lambda^kT_b^*B$-valued $r$-form over each fiber $E_b$. We define $I_\pi\colon \Omega^{k+r}(E) \to \Omega^k(B)$ by
\[
 \quad I_\pi\omega(b)=\int_{E_b}\Phi_\omega,
\]
where the fiber integral is computed with respect to the orientation induced by the bundle. Equivalently, the value of the $k$-differential form $I_\pi\omega$ on a $k$-tuple of vectors $\xi_1,\ldots, \xi_k \in T_bB$ can be computed by integrating the differential form $\omega^{b; \xi_1\wedge \ldots \wedge \xi_k} \in \Omega^{r}(E_b)$ defined by
    \begin{equation} \label{eq:retr}
        \omega^{b; \xi_1\wedge \ldots \wedge \xi_k}(v_1,\ldots, v_r) = \omega_e(\bar\xi_1,\ldots,\bar\xi_k,v_1,\ldots,v_r), \qquad v_1,\ldots, v_r \in T_e(E_b),
    \end{equation}
for $\bar\xi_i \in T_eE$ chosen as above.\medskip

We record here the local expression for the differential form $I_\pi\omega$. In the case $E=B \times F$ and $\pi$ is the trivial (product) bundle, we note that locally, in coordinate systems $x=(x^1,\ldots, x^n)$ and $y=(y^1,\ldots, y^r)$ for $B$ and $F$, respectively, every $(k+r)$-form $\omega$ can be written as a sum of forms of types:
    \begin{equation}
        \eta=a^{I,J}(x,y) dx^I \wedge dy^J,
    \end{equation}
for $p$- and $q$-multi-indexes $J=(j_1,\ldots, j_p)$ and $K=(k_1,\ldots, k_q)$ with $1\leq j_1<\ldots<j_p\leq n$, $1\leq k_1 < \ldots < k_q \leq r$ and $p+q=k+r$. Since $dx^i(v)=0$ for any vertical vector $v \in (V_E)_{(x,y)} = T_{(x,y)}(E_x) = T_{(x,y)}(\{x\}\times F)$, we have
\[
\left\langle \Phi_{\eta}(z;v_1\wedge \ldots \wedge v_r),D\pi(z)\xi_1\wedge\ldots\wedge D\pi(z)\xi_k\right\rangle = a^{I,J}(x,y) dx^I \wedge dy^J(\xi_1,\ldots,\xi_k,v_1,\ldots,v_r)
=0,\]
whenever $q=|J|<r$, and hence $I_\pi\eta =0$ in this case. On the other hand, if $q=|J|=r$, then $p=k$ and 
\begin{align*}
    \left\langle \Phi_{\eta}(z;\cdot ),D\pi(z)\xi_1\wedge\ldots\wedge D\pi(z)\xi_k\right\rangle = a^{I,J}(x,y)dx^I(\xi_1\wedge \ldots \wedge \xi_k) dy^J,
\end{align*}
thus
\begin{equation} \label{eq:product}
I_\pi\eta = \left(\int_F a(x,y)\,dy^J\right)\ dx^I.
\end{equation}
Since $I_\pi\omega$ is additive with respect to $\omega$, this characterizes this map entirely.\medskip

Among the key properties of the map $I_\pi$, we emphasize:

\begin{enumerate}
    \item[(i)] $I_\pi$ commutes with exterior differentiation, that is $I_\pi(d\omega) = d(I_\pi\omega)$ for any differential form $\omega \in \Omega^{k+r}(E)$.
    \item[(ii)] $I_\pi$ satisfies the following \emph{Fubini-type} identity: for any $\omega \in \Omega^{n+r}(E)$,
\[
\int_B I_\pi\omega = \int_E \omega.
\]
\end{enumerate}

For a proof of these properties, we refer to \cite{GHV}, Section 7.13, Proposition X, and Section 7.14, Theorem I. In particular, under the assumption of compactness and orientability of both the bundle $E \to B$ and the base $B$, we conclude that $I_\pi$ induces a morphism $H^{k+r}_{dR}(E) \to H^{k}_{dR}(B)$ between the De Rham cohomology groups which is an isomorphism for $k=\dim B$.

Now we estimate the comass of $I_\pi\omega$ under the assumption that $\pi$ is a Riemannian submersion:

\begin{lem} \label{lem:comass}
The comasses of $\omega$ and $I_\pi\omega$ are related by:
\[
\|I_\pi\omega(b)\| \leq \vol_{g_E}(E_b) \cdot \left(\sup_{e \in E_b} \|\omega(e)\|\right), \qquad \text{for all} \quad b \in B.
\]
\end{lem}

\begin{proof}
Let $\xi_1,\ldots, \xi_k \in T_bB$ be $g_B$-unit vectors. Since $\pi$ is a Riemannian submersion, for every $e \in E_b$, there are unique $\bar \xi_1,\ldots, \bar \xi_k \in (T_eE_b)^\perp$ such that $D\pi(e)\bar\xi_i=\xi_i$. Now choose a positive orthonormal (with respect to $g_E$) basis $\{v_j\}_{j=1}^r$ for $T_eE_b$, with duals forms $\{v_j^*\}$ in $T_e^*E_b$. Then
\[
v_1^* \wedge \ldots \wedge v_r^* = (\mathrm{d\,vol}_{E_b})_e,
\]
where $d\,\vol_{E_b}$ denotes the volume form of the fiber $E_b$ with respect to the metric induced by $g_E$ and the orientation induced by the bundle $\pi$. Using the notation from \eqref{eq:retr}, we get
\[
\omega^{b;\xi_1\wedge \ldots\wedge \xi_k} = \omega^{b;\xi_1\wedge \ldots\wedge \xi_k}(v_1,\ldots, v_r) v_1^*\wedge\ldots \wedge  v_r^* = \omega_e(\bar\xi_1,\ldots,\bar\xi_k,v_1,\ldots,v_r)(\mathrm{d\,vol}_{E_b})_e.
\]
On the other hand, using $|\bar\xi_1\wedge\ldots\wedge \bar\xi_k\wedge v_1\wedge \ldots\wedge v_r| = 1$, we see that $|\omega_e(\bar\xi_1,\ldots,\bar\xi_k,v_1,\ldots,v_r)|\leq \|\omega(e)\|$ and thus obtain the estimate
    \[
    |I_\pi\omega(b)(\xi_1,\ldots,\xi_k)|=\left|\int_{E_b}\omega^{b;\xi_1\wedge \ldots\wedge \xi_k}\right|\leq \int_{E_b}\left(\sup_{e \in E_b} \|\omega(e)\|\right)\mathrm{d\,vol}_{E_b},
    \]
which completes the proof.
\end{proof}

\section{Fiber integration and upper bounds for the volume spectrum}
The fiber integration map has a dual $I_\pi^*\colon \mathcal{D}_k(B) \to \mathcal{D}_{k+r}(E)$ between the spaces of currents. It is defined by duality as usual:
\[
\left\langle I_\pi^* T, \omega \right\rangle = \left\langle T, I_\pi \omega \right\rangle.
\]
Similar maps have appeared in the context of integral geometry as a generalization of the Cartesian product of currents \cite{simonbook}, most notably in the work of J. Brothers \cite{Brothers}, where it is called the \emph{lifting map}, and J.H. Fu \cite{Fu}, who defined a more general fiber product map with $G$-invariant currents in the fiber, for a $G$-bundle $\pi\colon E \to B$. 

In connection with these previous works, we observe:

\begin{lem} \label{lem:local}
The map
    \[
    I_\pi^* \colon \mathcal{D}_k(B) \cap \{\mathbf{M}(T) < \infty\} \to \mathcal{D}_{k+r}(E)
    \]
is a linear map such that:
\begin{enumerate}
    \item[(i)] If $E=B\times F$, then
        \[
        I_\pi^*T = T \times [\![ F]\!].
        \]
    \item[(ii)] if $\Phi\colon E \to \tilde E$ is a bundle map with $\varphi \circ \pi = \tilde \pi \circ \Phi$, where $\tilde \pi \colon \tilde E \to \tilde B$ is another fiber bundle, $\varphi \colon B \to \tilde B$ is smooth and $\Phi_b\colon E_b \to \tilde E_{\varphi(b)}$ is an orientation preserving diffeomorphism, then
        \[
        \Phi_{\#}I_\pi^*T = I_{\tilde \pi}(\varphi_{\#}T), \qquad \text{for all} \ T \in \mathcal{D}_k(B) \cap \{\mathbf{M}(T) < \infty\}.
        \]
\end{enumerate}
\end{lem}

\begin{proof}
Property (i) follows directly from the local expression for $I_\pi\eta$ derived in \eqref{eq:product} (see also \cite[Sections 7.12]{GHV}) and the definition of the Cartesian product, see \cite[Definition 26.16]{simonbook}. For (ii), it suffices to note that 
    \[
        I_\pi(\Phi^*\beta) = \varphi^*(I_{\tilde \pi}\beta),
    \]
for any $\beta \in \Omega^{k+r}(\tilde E)$ (see \cite[Proposition VIII]{GHV} for a proof).
\end{proof}

Intuitively, the first property above means that $I_\pi^*$ can be described locally as taking the product with the fiber. As a corollary (see also \cite{Brothers}), we get:

\begin{coro}
The map $I_\pi^*$ preserves the spaces of rectifiable and integral currents and chains. Concretely, there are well-defined maps:
\[
I_\pi^* \colon \mathcal{R}_k(B) \to \mathcal{R}_{k+r}(E),
\]
\[
I_\pi^* \colon \mathbf{I}_k(B) \to \mathbf{I}_{k+r}(E),
\]
and
\[
I_\pi^* \colon \mathcal{Z}_k(B) \to \mathcal{Z}_{k+r}(E).
\]
\end{coro}

\begin{proof}
The fact that $I_\pi$ preserves rectifiable currents (with integer densities) follows directly from the previous Lemma together with Theorem 3.3 and Corollary 3.5 from \cite{Brothers}, relying on Federer-Fleming's characterization of rectifiable currents, see \cite[4.1.28]{Federer}. Since $I_\pi^*(\partial T) = \partial(I_\pi^*T)$, we readily see that $T \in \mathbf{I}_k(B)$ implies $I_\pi^*T \in \mathbf{I}_{k+r}(E)$, and $T \in \mathcal{Z}_k(B)$ implies $I_\pi^*T\in \mathbf{I}_k(B)$.
\end{proof}

We also observe:

\begin{lem}
The currents $[\![E]\!]$ and $[\![B]\!]$ induced by integration over the total space and over the base, respectively, are related by $I_\pi^*[\![B]\!]=[\![E]\!]$.
\end{lem}

\begin{proof}
This follows readily from the Fubini-type property for $I_\pi$:
\[
    \left\langle I_\pi^*[\![B]\!], \omega\right\rangle = \left\langle  [\![B]\!] , I_\pi \omega \right\rangle  = \int_B I_\pi\omega = \int_E\omega = \left\langle[\![E]\!], \omega \right\rangle
.\qedhere\]\end{proof}

We can leverage the comass bounds for $I_\pi\omega$ derived in the previous section to obtain upper bounds for the mass of $I_\pi^*T$:

\begin{lem} \label{lem:mass}
Let $T \in \mathcal{D}_k(B)$. Then
    \[
    \mathbf{M}(I_\pi^*T) \leq \mathbf{M}(T) \cdot \left( \sup_{b \in B} \vol_{g_E}(E_b) \right).
    \]
\end{lem}

\begin{proof}
Let $\omega \in \Omega^{k+r}(E)$. Then
\begin{align*}
    |I_\pi^*T(\omega)| = |T(I_\pi \omega)| &\leq  \mathbf{M}(T) \cdot \sup_{b \in B}\|I_\pi\omega(b)\| \\
    & \leq \mathbf{M}(T)\cdot \sup_{b \in B} \left(\vol_{g_E}(E_b) \cdot \left(\sup_{e \in E_b} \|\omega(e)\|\right)\right)\\
    & \leq \mathbf{M}(T)\cdot \left( \sup_{b \in B} \vol_{g_E}(E_b)\right)\cdot \sup_{e \in E} \|\omega(e)\|,
\end{align*}
where we used Lemma \ref{lem:comass} to obtain the second inequality. Therefore
    \[
    \mathbf{M}(I_\pi^*T) = \sup\left\{ I_\pi^* T(\omega) \colon \omega \in \Omega^{k+r}(E), \sup_{e \in E}\|\omega(e)\| \leq 1 \right\} \leq \mathbf{M}(T)\cdot \left( \sup_{b \in B} \vol_{g_E}(E_b)\right). \qedhere
    \]
\end{proof}

\begin{lem} \label{lem:continuous}
The map
\[
I_\pi^* \colon \mathbf{I}_k(B) \to \mathbf{I}_{k+r}(E)
\]
is Lipschitz with respect to the topologies induced by $\mathcal{F}$, $\mathbf{M}$, and $\mathbf{F}$, with the same Lipschitz constant, depending only on the volume of the fibers of $\pi \colon E \to B$.
\end{lem}

\begin{proof}
The Lipschitz property for the mass norm is a consequence of the linearity of $I_\pi^*$ and Lemma \ref{lem:mass}. Using again the linearity of $I_\pi^*$ and the fact that $I_\pi^*$ preserves the space of integer multiplicity rectifiable currents, we see that $I_\pi^*$ is Lipschitz with respect to the flat metric, with the same Lipschitz constant. Finally, since for any two $k-$ currents $T,S \in \mathcal{R}_k(M)$ we have (see \cite[p. 66]{Pitts})
    \[\mathbf{F}(|S|,|T|) \leq \mathbf{M}(S-T),\]
the Lipschitz property for the $\mathbf{F}$-metric follows directly.
\end{proof}

Finally, we show similar mass bounds and Lipschitz continuity for modulo $2$ currents and their respective norms.

\begin{prop} \label{mod 2 fiber integration}
The map $I_\pi^*$ preserves congruence classes of rectifiable modulo $2$ currents. In particular, there are well-defined maps
    \[
    I_\pi^*\colon \mathbf{I}_k(B;\Z_2) \to \mathbf{I}_{k+r}(E;\Z_2),
    \]
and
    \[
    I_\pi^*\colon \mathcal{Z}_k(B;\Z_2) \to \mathcal{Z}_{k+r}(E;\Z_2).
    \]
Furthermore, we have the modulo 2 mass bounds:
    \begin{equation} \label{mass_bound}
        \mathbf{M}^2(I_\pi^*T) \leq \left( \sup_{b \in B} \vol_{g_E}(E_b) \right)\, \mathbf{M}^2(T),
    \end{equation}
for any $T \in \mathcal{R}_k(B;\Z_2)$, and the maps above are Lipschitz continuous with respect to the topologies induced by $\mathcal{F}^2$, $\mathbf{M}^2$ and $\mathbf{F}^2$, with the same Lipschitz constant from Lemma \ref{lem:continuous}.
\end{prop}

\begin{proof}
If $T_1,T_2 \in \mathcal{R}_k(B)$ are congruent modulo $2$ then, by the linearity of $I_\pi^*$, so are $I_\pi^*T_1$ and $I_\pi^*T_2$. This shows that $I_\pi^*$ induces a map from $\mathcal{R}_k(B;\Z_2)$ to $\mathcal{R}_{k+r}(E;\Z_2)$. For a congruence class $\tau \in \mathbf{I}_{k}(B;\Z_2)$, namely such that $\tau \in \mathcal{R}_{k}(B;\Z_2)$ and $\partial \tau \in \mathcal{R}_{k-1}(B;\Z_2)$, we have $I_\pi^*\tau  \in \mathcal{R}_{k+r}(E;\Z_2)$ and, by the commutativity of $\partial $ and $I_\pi ^*$,
\[
\partial(I_\pi^*\tau) = I_{\pi}^*(\partial \tau) \in \mathcal{R}_{k-1+r}(E;\Z_2).
\]
Therefore $I_\pi^*\tau \in \mathbf{I}_{k+r}(E;\Z_2)$. Since $\partial \tau =0$ implies $\partial(I_\pi^*\tau) = I_\pi^*(\partial \tau)=0$, this finishes the first part of the proof.

Now given $T \in \mathcal{R}_k(B)$, if $T=R+\partial S + 2Q$ for rectifiable $R$ and $S$ and a flat chain $Q$, then 
\[
I_\pi^*T = I_\pi^*R + \partial (I_\pi^*S) + 2 (I_\pi^*Q), \qquad \text{hence} \qquad \mathcal{F}^2(I_ \pi^*T) \leq \mathbf{M}(I_\pi^*R) + \mathbf{M}(I_\pi^*S)
.\]
By using Lemma \ref{lem:mass} and taking the infimum over $R$, $S$ and $Q$ and, we conclude
\[
\mathcal{F}^2(I_ \pi^*T) \leq C\cdot \mathcal{F}^2(T), \qquad \text{where} \quad C=\left( \sup_{b \in B} \vol_{g_E}(E_b) \right).
\]
Therefore, the map induced by $I_\pi^*$ in modulo 2 chains is Lipschitz, with the same Lipschitz constant as in Lemma \ref{lem:mass}. 

Now we show the mass bound:
\[
\mathbf{M}^2(I_\pi^*T) \leq C\, \mathbf{M}^2(T) \quad \text{for} \ T \in \mathcal{R}_k(B),
\]
and $C$ as above. Let $\delta>0$ be arbitrary, it follows from the Definition \ref{mod v mass} of $\mathbf{M}^2(T)$ that for any $\epsilon>0$, there exists $R \in \mathcal{R}_k(B)$ with
\[
\mathcal{F}^2(T-R) < \frac{\epsilon}{C} \qquad \text{and} \qquad \mathbf{M}(R) \leq \left( \mathbf{M}^2(T)+\frac{\delta}{C}\right)+\frac{\epsilon}{C},
\]
where we choose $t = \mathbf{M}^2(T)+\frac{\delta}{C}$. Then $I_\pi^*R \in \mathcal{R}_{k+r}(E)$ satisfies
\[
\mathcal{F}^2(I_\pi^*T - I_\pi^*R) =\mathcal{F}^2(I_\pi^*(T-R)) \leq C\mathcal{F}^2(T-R)< \epsilon,
\]
and
\[
\mathbf{M}(I_\pi^*R) \leq C\mathbf{M}(R) \leq C\mathbf{M}^2(T) + \delta + \epsilon.
\]
By the definition of the modulo 2 mass, see \eqref{mod v mass}, this shows that $\mathbf{M}^2(I_\pi^*T) \leq C\mathbf{M}^2(T) + \delta$. Since $\delta>0$ is arbitrary, this concludes the proof of \eqref{mass_bound}. The Lipschitz properties for the maps induced by $I_\pi^*$ in modulo 2 currents then follow exactly as in the proof of Lemma \ref{lem:continuous}.
\end{proof}\smallskip

\begin{prop} \label{fiber integrated sweepout}
Let $\Phi \colon X \to \mathcal{Z}_{n -1}(B;\Z_2)$ be a $p$-sweepout of $B$. Then
\[I_\pi^*\circ \Phi \colon X \to \mathcal{Z}_{n+r-1}(E;\Z_2)\]
is a $p$-sweepout of $E$.  
\end{prop}

The proof follows by combining the characterization of $p$-sweepouts in Marques-Neves \cite{MarquesNevesRicci} with the description of the fundamental groups of these spaces of cycles in terms of the covering maps from $\mathbf{I}_n(B;\Z_2)=\mathcal{C}(B)$ and $\mathbf{I}_{n+r}(E;\Z_2) = \mathcal{C}(E)$. 

\begin{proof}
Recall that $I_\pi^*[\![B]\!] = [\![E]\!]$. If $\gamma$ is a representative of the generator of 
\[
\pi_1(\mathcal{Z}_{n-1}(B;\Z_2),0)\simeq H_{n}(B;\Z_2) \simeq Z_2,
\]
then $\gamma\colon [0,1] \to \mathcal{Z}_{n-1}(B;\Z_2)$ is a closed path with $\gamma(0)=0=\gamma(1)$ that lifts to some $\tilde \gamma\colon [0,1] \to \mathbf{I}_n(B;\Z_2)$ with $\tilde \gamma(0) = 0$ and $\tilde \gamma(1) = [\![B]\!]$. Hence, $I_\pi^* \circ \gamma \colon [0,1] \to \mathcal{Z}_{n+r -1}(E;\Z_2)$ is a closed path based at the zero cycle in $E$, and $I_\pi^*\circ \tilde \gamma \colon [0,1] \to \mathbf{I}_{n+r}(E;\Z_2)$ is a path with $I_\pi^*\circ \tilde \gamma(0) = I_\pi^*(0)=0$ and $I_\pi^*\circ \tilde \gamma(1) = I_\pi^* [\![B]\!] = [\![E]\!]$. Therefore, $I_\pi^* \circ \tilde \gamma$ is homotopically nontrivial. 

Since the fundamental group of $\mathcal{Z}_{n+r-1}(E;\Z_2)$ is also isomorphic to $\Z_2$, we conclude that $I_\pi^*$ induces an isomorphism between fundamental groups. In addition, since these fundamental groups are abelian, it also induces isomorphisms in 1st homology and 1st cohomology with $\Z_2$ coefficients (see e.g. \cite[Section 2.A]{Hatcher}).

Now since $\Phi \colon X \to \mathcal{Z}_{n-1}(B;\Z_2)$ is a $p$-sweepout of $B$, if we let $\lambda_B$ be the generator of $H^*(\mathcal{Z}_{n-1}(B;\Z_2);\Z_2)$ and similarly for $\lambda_E$, then $\lambda = \Phi^*(\lambda_B) \in H^1(X;\Z_2)$ satisfies $\lambda^p \neq 0 \in H^p(X;\Z_2)$. On the other hand, by the isomorphism described above, 
\[(I_\pi^*)^* \colon H^1(\mathcal{Z}_{n+r-1}(E;\Z_2);\Z_2) \to H^1(\mathcal{Z}_{n-1}(B;\Z_2);\Z_2)\]
maps $\lambda_E$ to the generator $\lambda_B$, and hence
\[(I_\pi^*\circ \Phi)^*(\lambda_E) = \Phi^*((I_\pi^*)^*\lambda_E) = \Phi^*(\lambda_B) = \lambda\]
has nonzero $p$-th cup power. 

Finally, we observe that $I_\pi^* \circ \Phi$ has no concentration of mass. Given $e \in E$, the geodesic balls $B_\rho^{(E,g_E)}(e)$ in $(E,g_E)$ are mapped onto geodesic balls $B^{(B,g_B)}_\rho(\pi(e))$ by $\pi$ -- this can be readily derived from the following observation: if $\gamma\colon [0,1] \to E$ is a piecewise smooth curve, then $\ell_{(B,g_B)}(\pi \circ \gamma) \leq \ell_{(E,g_E)}(\gamma)$. If we write $B_\rho=B_\rho^{(B,g_b)}(\pi(e))$ and $U=\pi^{-1}(B_\rho)$, then $\pi$ restricts to a submersion $\pi|U\,\colon U \to B_\rho$ and, for any rectifiable current $T$, we have
    \begin{align*}
    \|I_\pi^*T\|(B^{(E,g_E)}_\rho(e)) & \leq \|I_\pi^*T\|\left( U \right) = \mathbf{M}(I_\pi^*T \ \lfloor \ U ) =\mathbf{M}\left(I_{\pi|U}^* \left(T \ \lfloor \ B_\rho \right)\right) \\
    &\leq \left( \sup_{b \in B_\rho} \vol_{g_E}(E_b)  \right) \cdot \mathbf{M}\left(T \ \lfloor B_\rho \right) \leq \left( \sup_{b \in B_\rho} \vol_{g_E}(E_b)  \right) \cdot \|T\|(B_\rho^{(B,g_B)}(\pi(e)).
    \end{align*}
Here we used the naturality of the fiber integration map, Lemma \ref{lem:local}, and the mass bound from Lemma \ref{lem:mass}. Since the constant above does not depend on $e$, $\rho$ or $T$, we conclude
    \begin{align*}
        & \lim_{\rho \to 0^+} \sup \left\{\, \|I_\pi^* \circ \Phi(x)\|(B_\rho^{(E,g_E)}(e)) \colon x \in X, \ e \in E\,\right\} \\
        & \qquad\leq \left( \sup_{b \in B_r} \vol_{g_E}(E_b)  \right) \cdot \lim_{\rho \to 0^+}\sup \left\{\, \| \Phi(x)\|(B_\rho^{(B,g_B)}(b)) \colon x \in X, \ b \in B\,\right\} = 0,
    \end{align*}
where we used that $\Phi$ has no concentration of mass.
\end{proof}

\begin{proof}[Proof of \ref{thm:width_fibration}]
Let $\Phi \in \mathcal{P}_p(B)$ be a $p$-sweepout for $(B,g_B)$ defined on some cubical complex $X$. By Proposition \ref{fiber integrated sweepout}, the map $I_\pi^* \circ \Phi \colon X \to \mathcal{Z}_{\dim E -1}(E;\mathbf{F};\Z_2)$ is a $p$-sweepout of $(E,g_E)$. Moreover, by Proposition \ref{mod 2 fiber integration}, we have
    \[\mathbf{M}(I_\pi^*(\Phi(x)) \leq \left( \sup_{b \in B} \vol_{g_E}(E_b) \right)\, \mathbf{M}(\Phi(x)),\]
for all $x \in X$, where we write $\mathbf{M}^2$ as $\mathbf{M}$ for simplicity. Thus,
    \[\omega_p(E,g_E) \leq \sup_{x\in X}\mathbf{M}(I_\pi^*(\Phi(x)) \leq \left( \sup_{b \in B} \vol_{g_E}(E_b) \right)\, \sup_{x \in X}\mathbf{M}(\Phi(x)).\]
Since $\Phi$ is an arbitrary $p$-sweepout for $(B,g_B)$, by taking the infimum over $\Phi$, we conclude
    \[\omega_p(E,g_E) \leq \left( \sup_{b \in B} \vol_{g_E}(E_b) \right)\,\omega_p(B,g_B).\qedhere\]
\end{proof}

\section{Upper bounds for the Allen-Cahn spectrum} \label{allen-cahn}

\subsection{Phase transition spectrum of Riemannian fibrations} Assume again $\pi\colon E \to B$ is a fiber bundle with compact base space $B$ and compact fiber $F$, and that $E$ and $B$ are endowed with Riemannian metric $g_E$ and $g_B$, respectively, which make $\pi$ a Riemannian submersion. We recall the following Fubini-type formula \cite[Chapter II, Theorem 5.6]{Sakai}:
\begin{equation} \label{fubini}
    \int_E f\,d\mu_{g_E} = \int_B \left( \int_{E_b} f\, d\mu_{E_b}\right) \,d\mu_{g_B},
\end{equation}
where $\mu_{g_E}$ and $\mu_{g_B}$ are the respective Riemannian volume measures, and $\mu_{E_b}$ is the volume measure corresponding to the metric induced by $g_E$ in the fibers $E_b=\pi^{-1}(b)$.

\begin{lem}
The map
\[
f \in W^{1,2}(B) \mapsto f\circ \pi \in W^{1,2}(E)
\]
is a linear bounded map with
\begin{equation} \label{eq:l2estimate}
\|f\circ \pi\|_{L^2(E,g_E)}^2 \leq \left( \sup_{b \in B} \vol_{g_E}(\pi^{-1}(b)) \right) \|f\|_{L^2(B,g_B)}^2,
\end{equation}
and
\begin{equation} \label{eq:gradientestimate}
\|\nabla^{g_E}(f\circ \pi)\|_{L^2(E,g_E)}^2 \leq \left( \sup_{b \in B} \vol_{g_E}(\pi^{-1}(b)) \right) \|\nabla^{g_B}f\|_{L^2(B,g_B)}^2.
\end{equation}
\end{lem}

\begin{proof}
Given an $g_B$-orthonormal frame $\{E_j\}_{j=1}^{n}$ on an open set $U\subset B$, we can find a $g_E$-orthonormal frame $\{X_j\}_{j=1}^{n+r}$ such that $X_j$ and $E_j$ are $\pi$-related for $j=1,\ldots, n$, and $X_j(e) \in \ker D\pi(e) = T_e(E_{\pi(e)})$ for all $j=n+1,\ldots, n+r$ and $e \in E$. For any $f \in C^\infty(B)$, if we let $F=f \circ \pi$, then
\[
g_E(X_j,\nabla^{g_E}F) = X_j(f\circ \pi) = E_j(f) = g_B(E_j,\nabla^{g_B}f), \qquad \text{for} \ j=1,\ldots,n, 
\]
while
\[
g_E(X_j,\nabla^{g_E}F) = X_j(f\circ \pi) = 0, \qquad \text{for} \ j=n+1,\ldots, n+r, 
\]
for $f\circ \pi|_{E_b} = f(b)$ is constant and $X_j$ is tangent to the fiber in the latter case. Therefore, $\nabla^{g_E}F$ and $\nabla^{g_B}f$ are $\pi$-related and
\[
|\nabla^{g_E}F|_{g_E} \circ \pi = |\nabla^{g_B}f|_{g_B}.
\]
Then \eqref{eq:gradientestimate} follows from:
\begin{align}\label{Dirichlet norm bundle}
\|\nabla^{g_E}F\|_{L^2(E,g_E)}^2 &= \int_E|\nabla^{g_E}F|_{g_E}^2\,d\mu_{g_E} = \int_B \left( \int_{E_b} |\nabla^{g_B}f|_{g_B}^2\, d\mu_{E_b}\right) \,d\mu_{g_B} \nonumber\\
&= \int_B \vol_{g_E}(E_b)\cdot|\nabla^{g_B}f|_{g_B}^2\,d\mu_{g_B} \leq \left( \sup_{b \in B} \vol_{g_E}(E_b)\right)\|\nabla^{g_B}f\|_{L^2(B,g_B)}^2.
\end{align}
Since \eqref{eq:l2estimate} follows directly from Fubini's formula above, by the density of smooth functions in $W^{1,2}$, this finishes the proof. 
\end{proof}

Now we are ready to bound the Allen-Cahn width of the fiber bundle by the product of the Allen-Cahn width of the base and the volume of the largest fiber. 

\begin{proof}[Proof of \ref{thm:phasetransition}]

Let $\Phi \colon X \to W^{1,2}(B)$ be a $p$-sweepout for $E_\e$ on $(B,g_B)$ with $\sup_{B}|\Phi(x)|\leq 1$ for all $x \in X$. Define a new family of functions $\Psi \colon X \to W^{1,2}(E)$ by 
\[
\Psi(x) = \Phi(x) \circ \pi, \qquad x \in X.
\]
By the observations above, $\Psi$ is continuous map. Moreover, by the $\Z_2$ equivariance of $\Phi$ and the linearity of $u\mapsto u\circ \pi$, we see that $\Psi$ is a $p$-sweepout in $E$. Moreover, by combining the Fubini theorem and the gradient estimate \eqref{eq:gradientestimate}, we obtain:
\[
E_\e(\Psi(x),g_E) \leq \left( \sup_{b \in B} \vol_{g_E}(E_b)\right) E_\e(\Phi(x),g_B) \leq \left( \sup_{b \in B} \vol_{g_E}(E_b)\right) \cdot c_\epsilon(p; B,g_B),
\]
for any $x \in X$. By taking the supremum over $x$ and the infimum over $p$-sweepouts $\Phi$, we conclude:
\begin{equation}
c_\epsilon(p;E,g_E) \leq \left( \sup_{b \in B} \vol_{g_E}(E_b)\right) \cdot c_\epsilon(p; B,g_B).
\end{equation}

\end{proof}

\section{The isoperimetric profile of sphere bundles}

Let $\pi\colon(E,g_E) \to (B,g_B)$ be a Riemannian fibration, with totally geodesic fibers $E_b = \pi^{-1}(b)$ isometric to the round sphere $S^r$ of dimension $r$. In this section, we would like to compare the isoperimetric profiles of $E$ and $B\times S^r$. In view of Lemma \ref{leastarea_isoperimetric}, in some concrete examples, this can be used to obtain lower bounds for the first width of $E$. 

The comparison result for isoperimetric profiles stated below has previously appeared in the work of Morgan-Howe-Harman \cite[Propositions 5 and 8]{MHH} (see also Ros \cite{Ros} and Morgan \cite{Morgan07}), where the authors develop a symmetrization method for subsets of warped products which intersect almost every fiber transversely. A closely related fiberwise symmetrization method was recently employed by C. Sung \cite{SungSymmetrization,Sung} to obtain lower bounds for the first eigenvalue of the Laplace operator on a Riemannian bundle $E \to B$ with totally geodesic fibers in terms of the first eigenvalue of the product of $B$ with those fibers. This class of manifolds was studied in depth in the classical work by Bergery-Bourguignon \cite{BBLaplacian}.

We will consider metrics on $E$ with fibers dilated by a constant factor. Concretely, let $H = (\ker D\pi) ^\perp\subset TE$ be horizontal distribution associated to $\pi$. The Riemannian metrics $g_E^\rho$ on $E$ defined by replacing $g_E|_{H^\perp}$ by $\rho^2 g_E|_{H^\perp}$, for each $\rho>0$, are called \emph{canonical variations} \cite{BBLaplacian} of $g_E$ and define a one-parameter family of Riemannian submersions $\pi \colon (E,g_E^\rho) \to (B,g_B)$.\medskip

We recall the notation from Subsection \ref{subsec:ac} concerning the Allen-Cahn equation and its connection to the isoperimetric problem. Additionally, if $X$ is a space with a right $G$-action and $\phi$ is a function defined on $X$, we write $\phi \cdot g$ to denote the function defined on $X$ by $(\phi\cdot g)(x)=\phi(x\cdot g)$.

In view of the connections between isoperimetric inequalities and the eigenvalues of the Laplacian \cite{Chavel}, we expect that Sung's strategy can also be used to obtain comparison results for isoperimetric profiles. In this section, we follow Sung's method to show:

\begin{thm} \label{thm:lower_bound}
Suppose that the double-well $W$ satisfies the following additional condition: $|W(s)| \to  \infty$ as $|s| \to \infty$. For any $m \in (-\vol(E,g_E), \vol(E,g_E))$, any $\rho>0$, and any $\e>0$, we have:
\[
    \inf\{E_\e(u) \colon u \in V_m(E,g^\rho_E) \} \geq \inf\{E_\e(v) \colon v \in V_m(B\times S^r(\rho)) \},
\]
where $B \times S^r(\rho)$ is endowed with the product metric. Consequently, the isoperimetric profiles of $(E,g_E)$ and $B\times S^r(\rho)$ satisfy
\[
I_{(E,g_E^\rho)} (t) \geq I_{(B\times S^r(\rho))} (t), \qquad \text{for all} \ t \in (0,\vol(E,g_E^\rho)).
\]
\end{thm}

\begin{proof}
We follow closely the proof of Theorem 1.1 in \cite{Sung}. By \cite{Hermann}, we see that $E\to B$ is a $G$-principal bundle, for the compact Lie group $G= SO(r+1)$ endowed with a bi-invariant Riemannian metric $g_G$. Furthermore, we can find a smooth $G$-principal bundle $\pi_P \colon P \to B$ such that $E$ is the associated bundle of $P$ with fiber $S^r$. The space $P$ is the pointwise union of the spaces of isometries $S^r \to \pi^{-1}(b)$, for $b \in B$ -- we refer to p. 239 in \cite{Hermann} for more details. In particular, there is a smooth quotient map $\tilde \pi \colon P \times S^r \to E$, and by \cite[Theorem 3.5]{Vilms}, we can endow $P$ with a Riemannian metric $g_P$  for which $\pi_P \colon (P,g_P) \to (B,g_B)$ is a Riemannian submersion.\smallskip

We endow $P \times S^r$ with the product metrics with fibers scaled by $\rho>0$, which we denote by $g_P+\rho^2 g_S$ (pulling back each metric to $P\times S^r$ by the natural projections from $P \times S^r$ onto $P$ and $S^r$). In general, the quotient map $\tilde \pi \colon (P\times S^r, g_P + \rho^2 g_S) \to (E,g_E^\rho)$ is \emph{not} a Riemannian submersion, nevertheless $T(P\times F)$ admits a $G$-invariant orthogonal decomposition as the sum of the tangent distribution to $\{p\} \times F$, the tangent distribution of the orbits of the $G$-action on the $P$ factor, and the horizontal spaces of the Riemannian bundle $\pi_P$ (lifted to $P\times F$); we refer to \cite[Lemma 3.2]{Sung} for more details.

In addition, one can locally define a $G$-invariant $(\dim G$)-form $\Omega_G$ on $P\times S^r$ which restricts to the volume of each fiber of $\pi_P\colon P \times S^r(\rho) \to B$ and which vanishes on the horizontal distribution of this submersion. These facts imply a Fubini-type formula for the quotient map $\tilde \pi \colon P\times S^r(\rho) \to E$ for functions of type $f \circ \tilde \pi$ and their Dirichlet energy. More precisely, \cite[Proposition 3.6]{Sung} shows that:
\begin{equation} \label{Fubini PxS}
\int_{P \times S^r(\rho)} (f\circ \tilde\pi) \,d\mu_{P \times S^r(\rho)} = \mathrm{vol}(G) \cdot \int_Ef \,d\mu_{g_E^\rho},
\end{equation}
for any $f \in L^1(E)$, and 
\begin{equation} \label{Fubini derivatives PxS}
\int_{P \times S^r(\rho)} |d(f\circ \tilde \pi) \wedge \Omega_G|^2_{P\times S^r(\rho)} \,d\mu_{P \times S^r(\rho)} = \mathrm{vol}(G) \cdot \int_E|\nabla f|^2_{g_\rho} \,d\mu_{g_E^\rho}.
\end{equation}
\smallskip

Let $u_{\epsilon,m}\in C^\infty(E)$ be a minimizer of $E_\e$ on $V_m(E,g_E^\rho)$, and let $\tilde u_{\e,m}=u_{\e,m} \circ \tilde \pi$ be its lift to $P\times S^r$. By our assumption on $W$, there exists $K=K(W,E,g_E,\rho)>0$ such that
\begin{equation} \label{L infinity estimate}
        \|\tilde u_{\e,m}\|_{L^\infty(P\times S^r(\rho))} = \|u_{\e,m}\|_{L^\infty(E,g_E^\rho)} \leq K,
\end{equation}
see e.g. \cite{GurtinMatano}.
Given $\delta>0$, we can find a $\delta$-close $C^2$ perturbation of $\tilde u_{\e,m}$ whose spherical symmetrization in each $S^r(\rho)$ factor is a Lipschitz function $\tilde u_{\e,m}^*$ defined on $P \times S^r(\rho)$ with controlled Allen-Cahn energy. In fact, by \cite[Theorem 3.9]{Sung} (with $c=1$, as the fibers of $E\to B$ are spheres; see also \cite[Chapter 7]{Baernstein}), we can choose $\tilde u_{\e,m}^*$ so that
    \begin{equation} \label{Dirichlet symmetrization}
    \int_{P \times S^r(\rho)} |d\tilde u_{\e,m}^*\wedge \Omega_G|^2\,d\mu_{P\times {S^r(\rho)}} \leq \int_{P\times S^r(\rho)} |d\tilde u_{\e,m}\wedge \Omega_G|^2\,d\mu_{P\times {S^r(\rho)}} + C_0\delta,
    \end{equation}
where the term involving $C_0=C_0(E,B,\pi,\rho)>0$ comes from the perturbation of $\tilde u_{\e,m}$, and
\begin{equation} \label{L1 estimate}
\int_{P\times S^r(\rho)}\left| \tilde u_{\e,m}^{*} -\tilde u_{\e,m}\right|\,d\mu_{P \times S^r(\rho)} \leq \vol(P,g_P)\cdot\rho^r \vol(S^r) \cdot \delta,
\end{equation}
from which we get
    \begin{equation} \label{potential symmetrization}
    \int_{P \times S^r(\rho)} W(\tilde u_{\e,m}^*)\,d\mu_{P\times {S^r(\rho)}} \leq \int_{P \times S^r(\rho)} W(\tilde u_{\e,m})\,d\mu_{P\times {S^r(\rho)}} + C_1 \delta,
    \end{equation}
where $C_1=C_1(W,E,B,\pi,\rho)>0$ and we used \eqref{L infinity estimate} to estimate $|W(\tilde u_{\e,m}^*)-W(\tilde u_{\e,m})|$ in terms of $\|\tilde u_{\e,m}^*-\tilde u_{\e,m}\|_{L^\infty}$ and $\|W'\|_{L^\infty([-K,K])}$. 

In general, the symmetrization $\tilde u_{\e,m}$ may not be $G$-invariant, but its orbit by the natural $G$-action in the space of functions remain $\delta$-close to $\tilde u_{\e,m}$ in the $C^2$ sense, namely (see \cite[Corollary 3.2]{SungSymmetrization})
    \begin{equation} \label{C2 equivariance}
        \sup_{g\in G}\|\tilde u_{\e,m}^*\cdot g - \tilde u_{\e,m}\|_{C^2(P\times S^r(\rho))}<\delta.
    \end{equation}

We also observe that, by \eqref{L1 estimate} and \eqref{L infinity estimate}, by adding a small constant to $\tilde u_{\e,m}^*$, we may assume that 
\begin{equation} \label{integral pullback}
\int_{P\times S^r(\rho)}\tilde u_{\e,m}^{*} \, d\mu_{P\times S^r(\rho)} = \int_{P\times S^r(\rho)}\tilde u_{\e,m} \, d\mu_{P\times S^r(\rho)} = \vol(G) \cdot m.
\end{equation}

We now average $\tilde u_{\e,m}^*$ over the fibers of the $G$ action on the first factor of $P \times S^r(\rho)$ (whose orbits are the fibers of the bundle $P \to B$) in order to define a $G$-invariant function on $P\times S^r(\rho)$ with controlled energy and which then passes to the $G$-quotient $B\times S^r(\rho)$. Concretely, define $\tilde u_{\e,m}^{\circledast}\colon P \times S^r(\rho)\to \R$ by
    \[
    \tilde u_{\e,m}^{\circledast}(p,x) = \frac{1}{\vol(G)} \int_G \tilde u_{\e,m}^*(p\cdot g,x) \,d\mu_{g_G}(g).
    \]
By the invariance of the metric on $G$, the function $u_{\e,m}^{\circledast}$ is $G$-equivariant, namely $u_{\e,m}^{\circledast}(p\cdot g,x) = u_{\e,m}^{\circledast}(p,x)$ for all $(p,x) \in P \times S^r(\rho), g \in G$. In addition, $u_{\e,m}^{\circledast}\in L^2(P\times S^r(\rho))$ (by \cite[Proposition 2.3]{Sung}) and, using \eqref{C2 equivariance}, we see that
    \begin{align*}
        |\tilde u_{\e,m}^{\circledast}(p,x)-\tilde u^*_{\e,m}(p,x)|& = \frac{1}{\vol(G)}\left|\tilde u^*_{\e,m}(p,x) \int_G\,d\mu_{g_G} - \int_G \tilde u_{\e,m}^*(p\cdot g,x) \,d\mu_{g_G}(g) \right| \\
        & \leq \frac{1}{\vol(G)} \int_G\left | \tilde u_{\e,m}(p,x)^* - (\tilde u^*_{\e,m}\cdot g)(p,x) \right|\,d\mu_{g_G} \leq \delta,
    \end{align*}
for a.e. $(p,x) \in P \times S^r(\rho)$,
    \begin{equation} \label{potential averaging}
    \int_{P\times S^r(\rho)}W(\tilde u_{\e,m}^{\circledast}) \,d\mu_{P\times S^r(\rho)} \leq \int_{P\times S^r(\rho)} W(\tilde u_{\e,m}^*)\,d\mu_{P\times S^r(\rho)} + C\delta,
    \end{equation}
where hereafter we denote by $C$ a positive constant depending only on $(W,E,B,\pi,\rho)$ (possibly changing from line to line). Moreover, by \cite[Lemmas 2.4 and 3.9]{Sung} we have $\tilde u_{\e,m}^{\circledast} \in W^{1,2}(P \times S^r(\rho))$ and
    \begin{equation} \label{Dirichlet averaging}
       \int_{P \times S^r(\rho)} |d\tilde u_{\e,m}^{\circledast}\wedge \Omega_G|^2\,d\mu_{P\times {S^r(\rho)}}\leq \int_{P \times S^r(\rho)} |d\tilde u_{\e,m}^*\wedge \Omega_G|^2\,d\mu_{P\times {S^r(\rho)}}.
    \end{equation}

By the $G$-invariance of $\tilde u_{\e,m}^{\circledast}$, we can use \cite[Proposition 2.1 and Lemma 3.7]{Sung} to see that there exists $v_{\e,m} \in W^{1,2}(B\times S^r(\rho))$ such that $v_{\e,m}(\pi_P(p),x) = \tilde u_{\e,m}^{\circledast}(p,x)$. In order to conclude the proof, we will show that 
    \[\inf_{V_m(B\times S^r(\rho))}E_\e \leq E_\e(v_{\e,m}) \leq E_\e(u_{\e,m}) + \frac{ C \delta}{\e\vol(G)},\] 
which implies the first stated inequality, by letting $\delta \downarrow 0$ and recalling that $u_{\e,m}$ is a minimizer. To show the first inequality, it suffices to prove that $v_{\e,m} \in V_m(B \times S^r(\rho))$. For this purpose, note that
    \begin{align*}
    \int_{P\times S^r(\rho)} \tilde u_{\e,m}^{\circledast}\,d\mu_{P\times S^r(\rho)} & = \frac{1}{\vol(G)} \int_G \int_{P\times S^r(\rho)}\tilde u_{\e,m}^*(p\cdot g,x) \,d\mu_{P\times S^r(\rho)}(p,x)\,d\mu_{g_G}(g) \\
    & = \frac{1}{\vol(G)} \int_G \int_{P\times S^r(\rho)}\tilde u_{\e,m}^*(p,x) \,d\mu_{P\times S^r(\rho)}(p,x)\,d\mu_{g_G}(g) = \vol(G)\cdot m,
    \end{align*}
where we used Fubini's Theorem, the fact that $G$ acts isometrically on $(P,g_P)$ and \eqref{integral pullback}. Hence, by Fubini's Theorem for the products $P \times S^r(\rho)$ and $B\times S^r(\rho)$ and for the bundle $P \to B$ and using that the fibers $P_b=\pi_P^{-1}(b)$ have volume $= \mathrm{vol}(G)$ for all $b \in B$, we get
    \begin{align*}
    \vol(G)\cdot m &= \int_{P\times S^r(\rho)} v_{\e,m}(\pi_P(p),x)\,d\mu_{P\times S^r(\rho)}(p,x)\\
    & = \int_{S^r(\rho)} \int_P v_{\e,m} (\pi_P(p),x)\,d\mu_{g_P}(p)\,d\mu_{S^r(\rho)}(x)\\
    & = \int_{S^r(\rho)}\left( \int_B\int_{P_b} v_{\e,m} (b,x)\,d\mu_{P_b}\,d\mu_{g_B}(b)\right)\,d\mu_{S^r(\rho)}(x) \\
    & = \vol(G) \int_{B\times S^r(\rho)} v_{\e,m}(b,x)\,d\mu_{B\times S^r(\rho)}(b,x).
    \end{align*}
Therefore $v_{\e,m} \in V_m(B \times S^r(\rho))$. 

We now estimate the energy of $v_{\e,m}$ in terms of the symmetrized and averaged functions. First, by arguing as in the previous computation (for $W(\tilde u_{\e,m}^{\circledast})$ and $W(v_{\e,m})$) and by using \eqref{potential averaging}, \eqref{potential symmetrization} and \eqref{Fubini PxS}, we get
    \begin{align*}
      \vol(G)\cdot \int_{B\times S^r(\rho)}W(v_{\e,m}) \,d\mu_{B\times S^r(\rho)} &=      \int_{P\times S^r(\rho)}W(\tilde u_{\e,m}^{\circledast}) \,d\mu_{P\times S^r(\rho)} \\
       & \leq  \int_{P\times S^r(\rho)} W(\tilde u_{\e,m}^*)\,d\mu_{P\times S^r(\rho)} + C\cdot \delta\\
       & = \int_{P\times S^r(\rho)} W(\tilde u_{\e,m})\,d\mu_{P\times S^r(\rho)} + (C+C_1)\cdot \delta\\
       & = \int_{P\times S^r(\rho)}\left( W(u_{\e,m})\circ \tilde\pi\right) \,d\mu_{P\times S^r(\rho)} + (C+C_1)\cdot \delta\\
       & =  \vol(G) \cdot \int_E W(u_{\e,m})\,d\mu_{g_E^\rho} + (C+C_1)\cdot \delta.
    \end{align*}
On the other hand, since $\tilde u_{\e,m}^{\circledast}$ is constant along the fibers of $\pi_P \times \mathrm{id}_{S^r(\rho)}$ and $\Omega_G$ restricts to volume forms of these fibers, we have
    \[
        |d\tilde u_{\e,m}^{\circledast}\wedge \Omega_G|_{P\times S^r(\rho)}^2 = |d\tilde u_{\e,m}^{\circledast}|^2_{P\times S^r(\rho)}.
    \]
Hence, using \eqref{Dirichlet norm bundle}, \eqref{Dirichlet averaging}, \eqref{Dirichlet symmetrization} and \eqref{Fubini derivatives PxS}
    \begin{align*}
          \vol(G)\cdot \int_{B\times S^r(\rho)}|\nabla v_{\e,m}|^2 \,d\mu_{B\times S^r(\rho)} &=      \int_{P\times S^r(\rho)}|\nabla \tilde u_{\e,m}^{\circledast}|^2_{P\times S^r(\rho)} \,d\mu_{P\times S^r(\rho)} \\
        & = \int_{P\times S^r(\rho)}|d\tilde u_{\e,m}^{\circledast}\wedge \Omega_G|_{P\times S^r(\rho)}^2\,d\mu_{P\times S^r(\rho)}\\
       & \leq  \int_{P\times S^r(\rho)} |d\tilde u_{\e,m}^*\wedge \Omega_G|^2\,d\mu_{P\times S^r(\rho)}\\
       &\leq \int_{P\times S^r(\rho)} |d\tilde u_{\e,m}\wedge \Omega_G|^2\,d\mu_{P\times {S^r(\rho)}} + C_0\delta \\
       & = \int_{P\times S^r(\rho)} |d(u_{\e,m}\circ \tilde \pi)\wedge \Omega_G|^2\,d\mu_{P\times {S^r(\rho)}} + C_0\delta \\
       & = \mathrm{vol}(G) \cdot \int_E|\nabla u_{\e,m}|^2_{g_\rho} \,d\mu_{g_E^\rho} + C_0\delta.
    \end{align*}
Therefore, 
    \[
    E_\e(v_{\e,m}) \leq E_\e(u_{\e,m}) + \frac{1}{\vol(G)}\left(\frac{\e}{2}C_0 +\frac{C_1+C}{\e}\right)\delta.
    \]

This proves the claimed energy estimate and concludes the proof of the comparison of the minimizers of the constrained Allen-Cahn energies. The comparison of the isoperimetric profiles then follows by taking the limit as $\e \downarrow 0$ and using the results of \cite{Modica} described in Subsection \ref{subsec:ac}.
\end{proof}

\begin{rmk}
The fiberwise symmetrization method explored by Sung can also be applied to compact Riemannian bundles with totally geodesic fibers isometric to $(F,g_F)$, provided that $g_F$ satisfies suitable Ricci curvature bounds. These bounds allow one to compare the isoperimetric profiles of the typical fiber with model spaces and derive energy estimates for such bundles. It is possible to use the strategy employed in this section to obtain more general comparison results in the spirit of the aforementioned work by Morgan-Howe-Harman \cite{MHH}, and also for the first Allen-Cahn and Almgren-Pitts widths. We plan to address this in our upcoming work.
\end{rmk}

\section{Applications}

\subsection{The width of Berger spheres} \label{subsec:berger}

Recall the one-parameter family of \emph{Berger metrics} on $S^3 = \{z=(z_1,z_2) \in \mathbb{C}^2 \colon |z_1|^2+|z_2|^2=1\}$:
\[
\left\langle v,w \right\rangle_{\tau} = \left\langle u,v \right\rangle - (1-\tau^2) \left\langle v,iz\right\rangle \left\langle w, iz\right\rangle, \qquad v,w \in T_zS^2,
\]
where $\left\langle \cdot, \cdot \right\rangle$ is the standard Euclidean metric in $\mathbb{C}^2$. For every $\tau \in \R_{>0}$, this defines a homogeneous metric on $S^3$ with $4$-dimensional isometry group, with $\left\langle\cdot,\cdot\right\rangle_1=\left\langle\cdot,\cdot\right\rangle=$ round metric on $S^3$. Moreover, the standard Hopf fibration $\pi \colon S^3 \to S^2(1/2)$ given by
\[
\pi(z_1,z_2) = \left(z_1\bar z_2, \frac{|z_1|^2-|z_2|^2}{2}\right)
\]
onto the $2$-sphere of radius $(1/2)$ is a Riemannian submersion with respect to $\left\langle\cdot,\cdot\right\rangle_\tau$ (see e.g. \cite[Section 2.1]{TorralboUrbano12}). Its fibers are closed geodesics in $(S^3,\left\langle\cdot,\cdot\right\rangle_\tau)$ of length $2\pi \tau$. 

Let $S^3_{\tau}$ denote $S^3$ with the Berger metric $\left\langle\cdot,\cdot\right\rangle_\tau$. The standard equatorial sphere $S^2=\{(z_1,z_2) \in S^3 \colon\ \mathrm{Im}z_2 =0\}$ -- and hence any other equatorial sphere, as $S^3_\tau$ is homogeneous -- and the Clifford torus $T=S^1(1/\sqrt{2})\times S^1(1/\sqrt{2})$ are minimal surfaces in $S^3_\tau$ for every $\tau$, and their area were explicitly computed by Torralbo \cite[Proposition 2]{Torralbo}. Moreover, Torralbo-Urbano \cite{TUIndex} showed that the equatorial spheres have Morse index 1 for any $\tau \in (0, 1]$, the Clifford tori have Morse index $1$ for any $\tau \in (0,1/\sqrt{3}]$, and Morse index $5$ for any $\tau \in (1/\sqrt{3}, 1)$. 

As an application of \ref{thm:width_fibration} and \ref{isoperimetric bounds}, we will prove \ref{corollary Berger}, which we recall here: for $\tau<\frac{1}{\pi}$, the first $3$ widths are achieved by the minimal Clifford tori in $S^3_\tau$, and that the equatorial spheres cannot achieve low $p$-widths $\omega_p(S^3_\tau)$ provided $\tau$ is sufficiently small. Note that this strongly contrasts with the Simon-Smith widths obtained by constraining the topology of the sweepout surfaces to be $2$-spheres. As pointed out in the introduction, this has recently been studied in connection with Zoll metrics on $S^3$ \cite{AmbrozioMarquesNeves1,AmbrozioMarquesNeves2}.\medskip

\begin{proof}[Proof of \ref{corollary Berger}]
By the computation of the $p$-widths of the $2$-sphere by Chodosh-Mantoulidis \cite{CMSurfaces}, we get
\[
\omega_p(S_\tau^3) \leq 2\pi\tau\cdot \omega_p(S^2(1/2)) = 2\pi\tau \cdot \frac{1}{2} \cdot 2\pi\lfloor\sqrt{p}\rfloor = 2\pi^2\tau\lfloor \sqrt{p} \rfloor.
\]
In particular,
\[
\omega_1(S^3_\tau) \leq \omega_2(S^3_\tau) \leq \omega_3(S^3_\tau) \leq 2\pi^2 \tau = \area_{S^3_\tau}(T),
\]
for any value of $\tau$. \medskip

We now show that $\omega_1(S^3_\tau)\geq \area_{S^3_\tau}(T)=2\pi^2\tau$, for $\tau \in(0,\frac{1}{\pi})$ using the result discussed in the previous section and the isoperimetric properties of $S^2(1/2)\times S^1(\tau)$. These properties were previously observed by Viana in \cite{Viana}, based on the results of \cite{MHH,PedrosaRitore}.

By Theorem \ref{thm:lower_bound}, we see that the isoperimetric profiles of $S^3_\tau$ and $S^2(1/2)\times S^1(\tau)$ satisfy
    \[
    I_{S ^3_\tau}(V) \geq I_{S^2(1/2)\times S^1(\tau)}(V), \qquad \text{for all} \quad V \in (0,\vol(S^3_\tau)).
    \]
From \cite[Theorem 4.3]{PedrosaRitore} (see the analysis on pages 1386-1388 of \cite{PedrosaRitore}) or also \cite[Theorem 23]{Ros}, we see that the torus $S^1(1/2) \times S^1(\tau)$ bounds an isoperimetric domain for half-volume in $S^2(1/2)\times S(\tau)$ provided $(0,\frac{1}{\pi})$. Hence, for this range of values for $\tau$,
    \begin{align*}
    I_{S^2(1/2)\times S^1(\tau)}\left({\textstyle \frac{1}{2}}\vol(S^3_\tau)\, \right)& =I_{S^2(1/2)\times S^1(\tau)}\left({\textstyle \frac{1}{2}}\vol(S^2(1/2)\times S^1(\tau))\, \right) \\
    & = \area_{S^3_\tau}(S^1(1/2)\times S^1(\tau))= \area_{S^3_\tau}(T).
    \end{align*}
    
On the other hand, $T$ bounds a half-volume domain in $S^3_\tau$, so $I_{S ^3_\tau}({\textstyle \frac{1}{2}}\vol(S^3_\tau))\leq \area_{S^3_\tau}(T)$, and the equality holds. Since $S^3_\tau$ contains no stable minimal surfaces for $\tau \in (0,2)$ (as it is simply connected and has positive Ricci curvature), it follows from Lemma \ref{leastarea_isoperimetric} that 
\[
\omega_1(S^2(1/2)\times S^1(\tau)) = \area_{S^3_\tau}(T) = 2\pi^2 \tau.
\]
Therefore,
\[
\omega_1(S^3_\tau) = \omega_2(S^3_\tau) = \omega_3(S^3_\tau) = \area_{S^3_\tau}(T)  = 2\pi^2 \tau, \qquad \text{for any} \ \tau \in \left(0, \frac{1}{\pi} \right).
\]

Concerning the minimal equatorial spheres in $S^3_\tau$, we observe that their area is bounded from below by $2\pi$ (see Proposition 2 in \cite{Torralbo}). Hence, for every $p \in \Z_{>0}$, if $\tau < \frac{1}{\pi\lfloor \sqrt{p}\rfloor}$, then 
\[
\vol_{S^3_\tau}(S^2)>\omega_p(S^3_\tau).
\]
\end{proof}

\subsection{Product of spheres} \label{width product spheres} In this subsection, we use $|\cdot|$ to denote the area functional.

As another application of (a special case of) \ref{thm:width_fibration}, we will use the characterization of low index minimal hypersurfaces in the product of spheres by Torralbo-Urbano \cite{TorralboUrbano} and Batista-Martins \cite{BM} (see also H. Chen \cite{ChenH}) and the solution of the half-volume isoperimetric problem in these spaces \cite{Ros}, in order to compute the low widths of product of spheres, as described in \ref{corollary product spheres}.\smallskip

We consider the Riemannian product of spheres $S^{n_1}\times S^{n_2}(a)$ with $n_1 \geq n_2 \geq 2$ and $a>0$. If we apply \ref{thm:width_fibration} for the projections onto each factor, we get the following upper bound for the $p$-widths:
\begin{equation} \label{upper bound spheres}
\omega_p(S^{n_1} \times S^{n_2}(a)) \leq \min\left\{\ \omega_p(S^{n_1}) \cdot a^{n_2}|S^{n_2}|\ , \ a^{n_2-1}\omega_p(S^{n_2})\cdot |S^{n_1}|  \ \right\}.
\end{equation}
Consider the totally geodesic (hence minimal) hypersurfaces in $S^{n_1} \times S^{n_2}(a)$:
\[
\Sigma_1=S^{n_1-1} \times S^{n_2}(a) \quad \text{and} \quad \Sigma_2=S^{n_1} \times S^{n_2-1}(a),
\]
We have the following area comparison between $\Sigma_1$ and $\Sigma_2$, which will be proved at the end of this subsection.

\begin{lem}\label{areacomparsion}
Let $a_0(n_1,n_2) = \frac{|S^{n_1}|/|S^{n_1-1}|}{|S^{n_2}|/|S^{n_2-1}|}$, where we recall $n_1\geq n_2\geq 2$. Then
    \[
    \sqrt{\frac{n_2-1}{n_1}}<a_0(n_1,n_2) <\sqrt{\frac{n_2}{n_1-1}}.
    \]
In particular, we have:
\begin{equation*}
\begin{aligned}
|\Sigma_1| < |\Sigma_2| &\qquad \text{when}\quad  a^2 \leq \frac{n_2-1}{n_1}; \\
|\Sigma_1| > |\Sigma_2| &\qquad \text{when}\quad a^2 \geq \frac{n_2}{n_1-1} .
\end{aligned}
\end{equation*}    
\end{lem}\medskip

In addition, we record the following computation that will be used in the proof:
    \begin{equation} \label{quotient areas Sigma12}
        \frac{|\Sigma_1|}{|\Sigma_2|} = \frac{|S^{n_1-1}|\cdot a^{n_2}|S^{n_2}|}{|S^{n_1}|\cdot a^{n_2-1}|S^{n_2-1}|} = a\cdot \frac{|S^{n_2}|/|S^{n_2-1}|}{|S^{n_1}|/|S^{n_1-1}|} =\frac{a}{a_0(n_1,n_2)}.
    \end{equation}\medskip

\begin{proof}[Proof of \ref{corollary product spheres}]
We will show that we can choose $a_-(n_1,n_2) := \sqrt{\frac{n_2-1}{n_1}}$ and $a_+(n_1,n_2)=\sqrt{\frac{n_2}{n_1-1}}$.\smallskip

Suppose $0<a \leq \sqrt{\frac{n_2-1}{n_1}}$. By Lemma \ref{areacomparsion}, we have $|\Sigma_1| < |\Sigma_2|$. Since
\[
\omega_{n_1+1}(S^{n_1}) \cdot a^{n_2}|S^{n_2}| = |S^{n_1-1}|\cdot a^{n_2}|S^{n_2}|=|\Sigma_1|,
\]
and
\[
a^{n_2-1}\omega_{n_1+1}(S^{n_2})\cdot |S^{n_1}| \geq a^{n_2-1}\omega_1(S^{n_2})\cdot |S^{n_1}| = a^{n_2-1}|S^{n_2-1}|\cdot |S^{n_1}| =|\Sigma_2| >|\Sigma_1|,
\]
by \eqref{upper bound spheres}, we have
\begin{equation} \label{width small radius}
\omega_{n_1+1}(S^{n_1}\times S^{n_2}(a)) \leq |\Sigma_1|.
\end{equation}
On the other hand, by Theorem \ref{widths are achieved}, the $1$-width $\omega_1(S^{n_1} \times S^{n_2}(a))$ is achieved by an integer combination of the areas of minimal hypersurfaces with optimal regularity and total index $\leq 1$. By \cite[Theorem A]{BM}, the only such minimal hypersurfaces are $\Sigma_1$ and $\Sigma_2$, for $a$ in this range. Therefore, we have $\omega_1(S^{n_1}\times S^{n_2}(a))\geq |\Sigma_1|$ and
\begin{equation} 
\omega_1(S^{n_1}\times S^{n_2}(a)) =\ldots= \omega_{n_1+1}(S^{n_1}\times S^{n_2}(a)) =|\Sigma_1|.
\end{equation}\medskip

Now consider $a \geq \sqrt{\frac{n_2}{n_1-1}}$. By Lemma \ref{areacomparsion} above, we have $|\Sigma_1| > |\Sigma_2|$. Now we can use
\[
a^{n_2-1}\omega_{n_2+1}(S^{n_2}) \cdot |S^{n_1}| = a^{n_2-1}|S^{n_2-1}|\cdot |S^{n_1}|=|\Sigma_2|,
\]
and
\[
\omega_{n_2+1}(S^{n_1}) \cdot a^{n_2}|S^{n_2}| \geq \omega_1(S^{n_1}) \cdot a^{n_2}|S^{n_2}| = |S^{n_1-1}|\cdot a^{n_2}|S^{n_2}|= |\Sigma_1| > |\Sigma_2|,
\]
to conclude
\begin{equation} \label{width large radius}
\omega_{n_2+1}(S^{n_1}\times S^{n_2}(a)) \leq |\Sigma_2|.
\end{equation}
Again, by Theorem \ref{widths are achieved}, the $1$-width $\omega_1(S^{n_1} \times S^{n_2}(a))$ is achieved by an integer combination of the areas of minimal hypersurfaces with optimal regularity and index $\leq 1$. But \cite[Theorem B]{BM} shows that the only such minimal hypersurfaces are $\Sigma_1$ and $\Sigma_2$. Thus, have $\omega_1(S^{n_1}\times S^{n_2}(a))\geq |\Sigma_2|$ and
\begin{equation} 
\omega_1(S^{n_1}\times S^{n_2}(a)) =\ldots= \omega_{n_2+1}(S^{n_1}\times S^{n_2}(a))=|\Sigma_2|.
\end{equation}\smallskip

Now we address the low dimensional cases, namely we assume hereafter $3\leq n_1+n_2\leq 7$ and prove the last statement in \ref{corollary product spheres}. By \cite[Corollary 1]{TorralboUrbano}, we see that $S^{n_1}\times S^{n_2}(a)$ admits no stable minimal hypersurfaces. Therefore, by combining the observation above and Lemma \ref{leastarea_isoperimetric}, we see that the least area minimal hypersurface in $S^{n_1}\times S^{n_2}(a)$ is given by either $\Sigma_1$ or $\Sigma_2$. Therefore,
    \[\omega_1(S^{n_1}\times S^{n_2}(a)) \geq \min\{|\Sigma_1|,|\Sigma_2|\}.\]
Again, inequalities \eqref{width small radius} and \eqref{width large radius} remain valid under the assumptions $|\Sigma_1|\leq |\Sigma_2|$ and $|\Sigma_1|\geq |\Sigma_2|$, respectively. Since these are equivalent to $0<a \leq a_0(n_1,n_2)$ and $a\geq a_0(n_1,n_2)$, we conclude that
    \[
    \omega_1(S^{n_1}\times S^{n_2}(a)) =\ldots= \omega_{n_1+1}(S^{n_1}\times S^{n_2}(a)) =|\Sigma_1|, \qquad \text{for} \quad a\in (0, a_0(n_1,n_2)]
    \]
and
    \[
    \omega_1(S^{n_1}\times S^{n_2}(a)) =\ldots= \omega_{n_2+1}(S^{n_1}\times S^{n_2}(a))=|\Sigma_2|, \qquad \text{for} \quad a\in [a_0(n_1,n_2),\infty).
    \]
which finishes the proof of \ref{corollary product spheres}.
\end{proof}

\begin{rmk}
In the case $n_1>n_2=1$, by \cite[Corollary 1]{TorralboUrbano} we see that the only connected, embedded, stable minimal hypersurfaces in $S^{n_1}\times S^{1}(a)$ are the totally geodesic slices $S^{n_1}\times \{q\}$, for $q \in S^1(a)$. In particular, any stable minimal hypersurface has area $\geq |S^{n_1}| = \frac{1}{2}|\Sigma_2|$. On the other hand, it follows from \cite[Theorem 23]{Ros} that the half-volume isoperimetric domains in $S^{n_1}\times S^{n_2}(a)$ correspond to the domain bounded by either $\Sigma_1$, if $|\Sigma_1|\leq |\Sigma_2|$, or by $\Sigma_2$, if $|\Sigma_2|\leq |\Sigma_1|$. 

If we let $a_0(n_1,1)=\frac{|S^{n_1}|/|S^{n_1-1}|}{\pi}$, then for any $a \in (0,\frac{1}{2}a_0(n_1,1)]$, equation \eqref{quotient areas Sigma12} shows that $|\Sigma_1| \leq \frac{1}{2}|\Sigma_2|$. Therefore, by Lemma \ref{leastarea_isoperimetric}, we see that $\Sigma_1$ is a least area minimal hypersurface and
    \[
    \omega_1(S^{n_1}\times S^{1}(a)) \geq |\Sigma_1|.
    \]
Using \ref{thm:width_fibration} for the projection $S^{n_1}\times S^1(a)\to S^{n_1}$ and the knowledge of the first $(n_1+1)$ widths of $S^{n_1}$, we conclude:
  \[
\omega_1(S^{n_1}\times S^{1}(a)) =\ldots= \omega_{n_1+1}(S^{n_1}\times S^{1}(a)) =|\Sigma_1|.
    \]
On the other hand, we observe that the geometry of $S^{n_1} \times S^1(a)$ has different features for large $a$ -- compare with Berger spheres $S^3_\tau$ with $\tau>1$, as observed in \cite{TUIndex}.
\end{rmk}\medskip

We observe that \ref{corollary product spheres} extends the width computation in \cite[Theorem C]{BM} in the following directions. First, by relying on the upper bound from \ref{thm:width_fibration} and on the knowledge of the first $n+1$ widths of the round sphere $S^n$, we can compute the first $(n_1+1)$ widths of $S^{n_1}\times S^{n_2}(a)$ for small $a$, and the first $(n_2+1)$ widths of $S^{n_1}\times S^{n_2}(a)$ for large $a$. Second, since our computation does not rely on the smoothness of the min-max minimal hypersurfaces that achieve the widths, we are able to extend it to any dimension $n_1+n_2 \geq 3$, provided $a$ lies in a suitable range (compare with \cite[Remark 6.2]{BM}). Lastly, in low dimensions, we bypass the the potential issue of the existence of other index 1 minimal hypersurfaces and compute the first few widths of $S^{n_1}\times S^{n_2}(a)$ for all values of $a>0$ by relying on the connection to the isoperimetric problem.\medskip

For completeness, we also mention the general upper bound for $n_2=2$:
\[
\omega_p(S^{n_1}\times S^{2}(a)) \leq |S^{n_1}|\cdot 2\pi  a\cdot \lfloor \sqrt{p}\rfloor,
\]
which follows from \eqref{upper bound spheres} and the computation of the $p$-widths of the $2$-sphere by Chodosh-Mantoulidis \cite{CMSurfaces}, and the specific bounds for $n_2=3$:
\[
\omega_7(S^{n_1}\times S^3(a)) \leq |S^{n_1}| \cdot a^2\cdot 2\pi^2 \quad \text{and} \quad \omega_{13}(S^{n_1}\times S^{3}(a)) < |S^{n_1}|\cdot 8a^2\pi 
\]
which follow from \eqref{upper bound spheres} and the computation of low widths of the $3$ sphere by Marques-Neves \cite{MarquesNevesWillmore} and more recently by Chu-Li \cite{CPChuLi}.

We now address the area comparison between $\Sigma_1$ and $\Sigma_2$. 

\begin{proof}[Proof of Lemma \ref{areacomparsion}]

Using the closed expression for the area of $S^n$ in terms of Gamma functions, we can compute
\[
a_0(n_1,n_2)=\frac{\Gamma(\frac{n_2+1}{2})/\Gamma(\frac{n_2}{2})}{\Gamma(\frac{n_1+1}{2})/\Gamma(\frac{n_1}{2})}.
\]
Hence, it suffices to show that
\begin{equation} \label{Gautschi}
\sqrt{\frac{n_2-1}{n_1}}<\frac{\Gamma(\frac{n_2+1}{2})/\Gamma(\frac{n_2}{2})}{\Gamma(\frac{n_1+1}{2})/\Gamma(\frac{n_1}{2})} <\sqrt{\frac{n_2}{n_1-1}}.
\end{equation}
This inequality follows from Gautschi's inequality (which uses the logarithmic convexity of the Gamma function, see e.g. \cite{Qi}):
\[\sqrt{\frac{n_2-1}{2}}<\Gamma(\frac{n_2+1}{2})/\Gamma(\frac{n_2}{2}) < \sqrt{\frac{n_2}{2}},\]
\[\sqrt{\frac{n_1-1}{2}}<\Gamma(\frac{n_1+1}{2})/\Gamma(\frac{n_1}{2}) < \sqrt{\frac{n_1}{2}},\]
which readily implies \eqref{Gautschi}. The conclusions about the comparison between the areas of $\Sigma_1$ and $\Sigma_2$ follow from computing the quotient of the areas of $\Sigma_1$ and $\Sigma_2$, see \eqref{quotient areas Sigma12} above.
\end{proof}

\bibliographystyle{amsalpha}
\bibliography{main}

\end{document}